\pdfoutput=1
\documentclass[ps,twoside]{imsart}
\RequirePackage{amscd,amsfonts,amsmath,amssymb,amsthm}
\RequirePackage{enumerate,latexsym,mathrsfs}
\RequirePackage{graphicx,epsfig,subfig}
\RequirePackage[usenames,dvipsnames]{color}
\RequirePackage[OT1]{fontenc}
\RequirePackage[numbers]{natbib}
\RequirePackage[colorlinks,linkcolor=blue,citecolor=blue,urlcolor=blue]{hyperref}
\RequirePackage{hypernat,url}
\RequirePackage{ifthen,xifthen}

\doi{10.1214/11-PS180}
\pubyear{2011}
\volume{8}
\issue{0}
\firstpage{155}
\lastpage{209}

\startlocaldefs
\usepackage{pkg}

\makeatletter
\def\url@leostyle{%
  \@ifundefined{selectfont}{\def\UrlFont{\sf}}{\def\UrlFont{\small}}}
\makeatother
\newcommand{\abholder}{MR1712629}
\newcommand{\absperc}{MR2073175}
\newcommand{\adapath}{PhysRevLett.83.1359}
\newcommand{\aizenmannumber}{MR1431856}
\newcommand{\aizenmansclim}{MR1635999}
\newcommand{\akn}{MR901151}
\newcommand{\bcksperc}{MR1868996}
\newcommand{\bchssi}{MR2155704}
\newcommand{\bchssii}{MR2165583}
\newcommand{\bchssiii}{MR2260845}
\newcommand{\bclperci}{MR2726646}
\newcommand{\bclpercii}{MR2726647}
\newcommand{\bhperc}{MR0091567}
\newcommand{\beffaradcopin}{arXiv:1107.0158}
\newcommand{\beffara}{MR2310300}
\newcommand{\beffarauniv}{MR2477376}
\newcommand{\beffarahdim}{MR2078552}
\newcommand{\beffaradim}{MR2435854}
\newcommand{\benjaminikozma}{arXiv:1105.2638}
\newcommand{\bk}{MR799280}
\newcommand{\bkperchd}{MR1261058}
\newcommand{\blpsaop}{MR1733151}
\newcommand{\blpsgafa}{MR1675890}
\newcommand{\bsvoronoi}{MR1646475}
\newcommand{\burtonkeane}{MR990777}
\newcommand{\chayeslei}{MR2337476}
\newcommand{\cmnw}{MR2322705}
\newcommand{\cmnwfull}{MR2249794}
\newcommand{\caratheodory}{MR1511737}
\newcommand{\cardy}{MR1151081}
\newcommand{\cardylec}{arXiv:math-ph/0103018}
\newcommand{\chelkaksmirnovuniv}{arXiv:0910.2045}
\newcommand{\chelkaksmirnovising}{chelkak.smirnov:ising}
\newcommand{\dcssaw}{arXiv:1007.0575}
\newcommand{\dcsmirnovsurvey}{arXiv:1109.1549}
\newcommand{\dklp}{RSA:RSA20342}
\newcommand{\dubedatwatts}{MR2226888}
\newcommand{\dubedatdual}{MR2118865}
\newcommand{\dubedatdualii}{MR2571956}
\newcommand{\epstein}{MR614728}
\newcommand{\fkg}{MR0309498}

\newcommand{\gsperc}{MR1675079}
\newcommand{\hshighdim}{MR1043524}
\newcommand{\hslace}{MR1283177}
\newcommand{\hssaw}{MR1356575}
\newcommand{\harrisineq}{MR0115221}
\newcommand{\harrislbd}{MR0115221}
\newcommand{\honglersmirnov}{arXiv:1008.2645}
\newcommand{\ikhlefrajabpourat}{MR2754707}
\newcommand{\kagernienhuis}{MR2065722}
\newcommand{\kakutani}{MR0014647}
\newcommand{\kestensq}{MR575895}
\newcommand{\kestenhd}{MR1064563}
\newcommand{\kenyon}{MR1782431}
\newcommand{\kenyongff}{MR1872739}
\newcommand{\kenyonloops}{arXiv:1105.4158}
\newcommand{\kszwords}{MR1637089}
\newcommand{\kwdimers}{MR2737268}
\newcommand{\lawlerparkcity}{MR2523461}
\newcommand{\lawlersheffield}{arXiv:0906.3804}

\newcommand{\llnhcap}{MR2576752}
\newcommand{\lpsa}{MR1230963}
\newcommand{\lswlerw}{MR2044671}
\newcommand{\lswonearm}{MR1887622}
\newcommand{\lswrestr}{MR1992830}
\newcommand{\lswintersecI}{MR1879850}
\newcommand{\lswintersecII}{MR1879851}
\newcommand{\lswintersecIII}{MR1899232}
\newcommand{\lwuniv}{MR1796962}
\newcommand{\millergl}{arXiv:1002.0381}
\newcommand{\milleruniv}{arXiv:1010.1356}
\newcommand{\millersheffield}{miller.sheffield}
\newcommand{\mnwcardy}{mendelson.nachmias.watson}
\newcommand{\mrloewner}{MR2163382}
\newcommand{\nachmias}{MR2570320}
\newcommand{\reimer}{MR1751301}
\newcommand{\rohdeschramm}{MR2153402}
\newcommand{\russo}{MR0488383}
\newcommand{\schramm}{MR1776084}
\newcommand{\schrammmadrid}{MR2334202}
\newcommand{\schrammwilson}{MR2188260}
\newcommand{\schrammperc}{MR1871700}
\newcommand{\schrammwatts}{arXiv:1003.3271}
\newcommand{\sheffieldtreescle}{MR2494457}
\newcommand{\ssexplorer}{MR2184093}
\newcommand{\ssdgff}{MR2486487}
\newcommand{\ssgff}{arXiv:1008.2447}
\newcommand{\skorohoddudley}{MR0230338}
\newcommand{\smirnovrci}{MR2680496}
\newcommand{\smirnovrcii}{smirnov:rc.ii}
\newcommand{\ksrciii}{kemppainen.smirnov:rc.iii}

\newcommand{\smirnovpercarxiv}{arXiv:0909.4499} 
\newcommand{\smirnovperc}{MR1851632}
\newcommand{\smirnovtowards}{MR2275653}
\newcommand{\smirnovicmplisbon}{MR2227824}
\newcommand{\smirnovicmhyderabad}{arXiv:1009.6077}
\newcommand{\swperc}{MR0494572}
\newcommand{\swcle}{arXiv:1006.2374}
\newcommand{\wernerparkcity}{MR2523462}
\newcommand{\wernerstflour}{MR2079672}
\newcommand{\wilsonxor}{arXiv:1102.3782}
\newcommand{\zhan}{MR2439609}
\newcommand{\zhanii}{MR2682265}
\newcommand{\ahlfors}{MR510197}
\newcommand{\ahlforsconf}{MR2730573}

\newcommand{\billingsleyc}{MR1700749}
\newcommand{\bollobasrg}{MR1864966}
\newcommand{\brperc}{MR2283880}
\newcommand{\grimmett}{MR1707339}
\newcommand{\grimmettprob}{MR2723356}
\newcommand{\kesten}{MR692943}
\newcommand{\ks}{MR1121940}
\newcommand{\lawler}{MR2129588}
\newcommand{\levy}{MR1188411}
\newcommand{\liebloss}{MR1817225}
\newcommand{\mpbm}{MR2604525}
\newcommand{\pommerenke}{MR1217706}
\newcommand{\ssf}{MR1970295}
\newcommand{\sscplx}{MR1976398}
\newcommand{\stauffer}{nla.cat-vn1787821}

\DeclareMathOperator{\capacity}{cap}
\DeclareMathOperator{\hcap}{hcap}

\DeclareMathOperator{\SL}{SL}

\newcommand{\site}{\text{\textup{s}}}
\newcommand{\bond}{\text{\textup{b}}}

\newcommand{\ann}[3]{A\left(#1;#2,#3\right)}
\newcommand{\bm}{W}
\newcommand{\clos}[1]{\overline{#1}}
\newcommand{\dN}{\mathcal{N}}
\newcommand{\doint}[1]{\oint_{#1}^{\text{\textup d}}}
\newcommand{\inter}[1]{#1^\circ}
\newcommand{\sle}{\text{SLE}}
\newcommand{\cle}{\text{CLE}}

\newcommand{\dhaus}{{d_\cH}} 
\newcommand{\dcurve}{{d_\cU}} 
\newcommand{\boxdim}{\ol{\dim}_B} 
\newcommand{\crv}[1]{\cC_{#1}} 
\newcommand{\crvset}[1]{\Om_{#1}} 

\newcommand{\cconv}{\rightsquigarrow}

\endlocaldefs

\begin{document}

\begin{frontmatter}
\title{Conformally invariant scaling limits in planar critical percolation\thanksref{T1}}
\thankstext{T1}{This is an original survey paper.}
\runtitle{Conformally invariant scaling limits in percolation}

\begin{aug}
\author{\fnms{Nike} \snm{Sun}\thanksref{t1}\ead[label=e1]{nikesun@stanford.edu}}

\address{Department of Statistics, Stanford University \\ 390 Serra Mall \\ Stanford, CA 94305\\
\printead{e1}}

\thankstext{t1}{Partially supported by Department of Defense (AFRL/AFOSR) NDSEG Fellowship.}
\runauthor{N. Sun}

\affiliation{Stanford University}
\end{aug}

\begin{abstract}
This is an introductory account of the emergence of conformal
invariance in the scaling limit of planar critical percolation. We give
an exposition of Smirnov's theorem (2001) on the conformal invariance
of crossing probabilities in site percolation on the triangular
lattice. We also give an introductory account of Schramm-Loewner
evolutions ($\mathrm{SLE}_{\kappa}$), a one-parameter family of conformally
invariant random curves discovered by Schramm (2000). The article is
organized around the aim of proving the result, due to Smirnov (2001)
and to Camia and Newman (2007), that the percolation exploration path
converges in the scaling limit to chordal $\mathrm{SLE}_6$. No prior knowledge
is assumed beyond some general complex analysis and probability theory.
\end{abstract}


\received{\smonth{8} \syear{2011}}

\end{frontmatter}

\tableofcontents

\section{Introduction}
\label{ch:intro}

The field of {\bf percolation theory} is motivated by a simple physical question: does water flow through a rock? To study this question, Broadbent and Hammersley \cite{\bhperc} developed in 1957 the following model for a random porous medium. View the material as an undirected graph $G$, a set of vertices $V$ joined by edges $E$. In the percolation literature the vertices are called {\bf sites} and the edges {\bf bonds}. In {\bf site percolation}, each site is independently chosen to be {\bf open} or {\bf closed} with probabilities $p$ and $1-p$ respectively (we refer to $p$ as the ``site probability''). The closed sites are regarded as blocked; the open sites, together with the bonds joining them, form the {\bf open subgraph}, through which water can flow. Percolation theory is the study of the connected components of the open subgraph, called the {\bf open clusters}. Fig.~\ref{fig:intro.open.subg} shows site percolation at $p=1/2$ on the triangular lattice $T$ (blue = open, yellow = closed), with the open subgraph
highlighted.\looseness=-1

Formally, a {\bf (site) percolation configuration} is a random element
$\om$ of $(\Om,\cF)$ where $\Om=\set{0,1}^V$ ($0=$ closed, $1=$ open) and $\cF$ is the usual product $\si$-algebra generated by the projections $\om\mapsto\om_i$. Let $\P_p\equiv\P^\site_p$ denote the probability measure on $(\Om,\cF)$ induced by percolation with site probability $p$.

\begin{figure}
\centering
\includegraphics[width=0.6\textwidth]{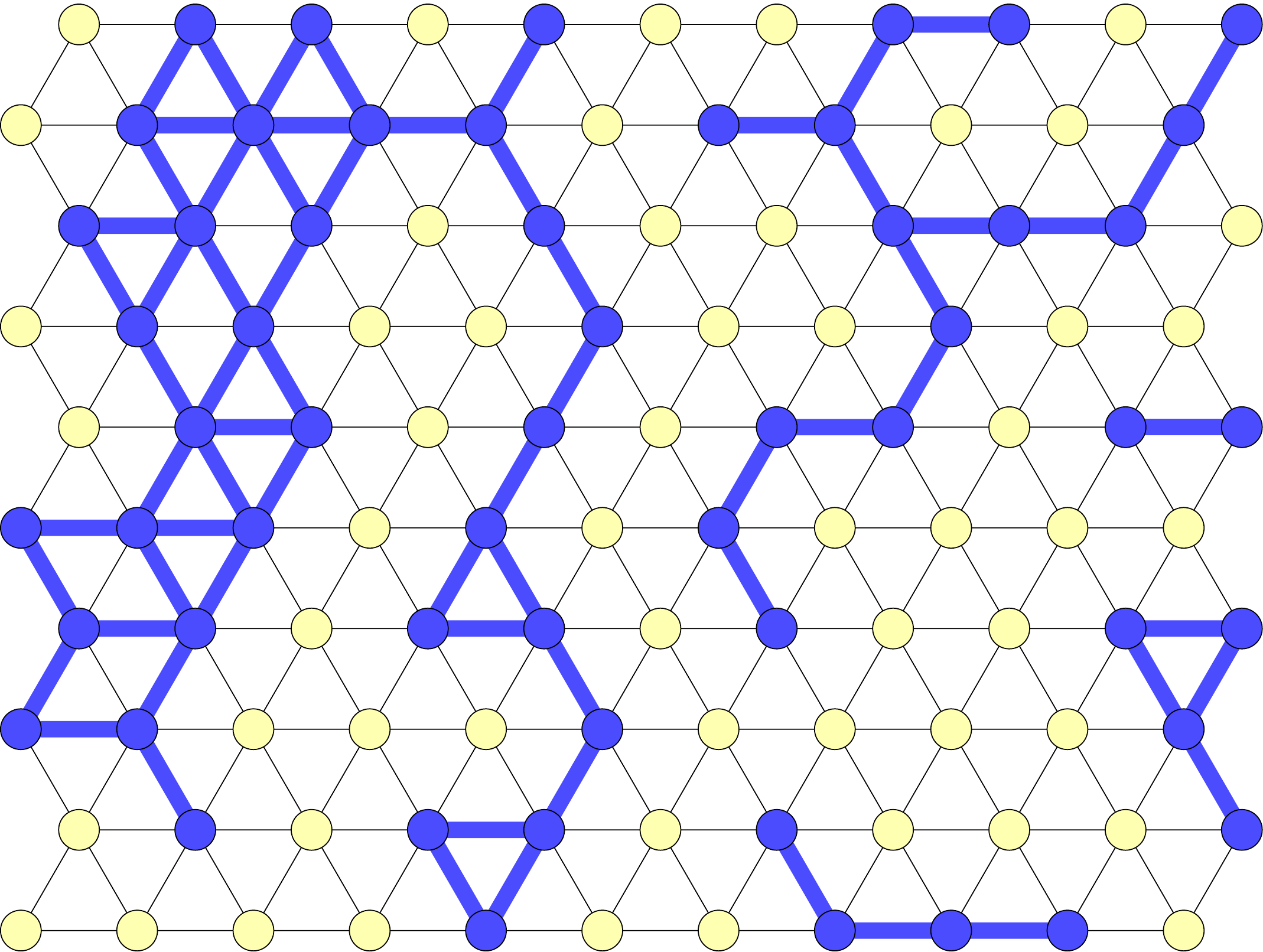}
\caption{Site percolation on $T$}
\label{fig:intro.open.subg}
\end{figure}

This basic model is of course open to any number of variations. An important one is {\bf bond percolation}, defined similarly except that edges (bonds) rather than vertices are chosen to be open independently with probability $p$, and the open subgraph is induced by the open edges. Denote the resulting probability measure on $\set{0,1}^E$ by $\P^\bond_p$. In the physical interpretation, in site percolation, water is held mainly in the sites or ``pockets'' and flows between adjacent pockets --- whereas in bond percolation, water is held mainly in the bonds or ``channels,'' and the \emph{sites} express the adjacency of channels. Historically, bond percolation has been studied more extensively than site percolation; however, bond percolation on a graph $G$ is equivalent to site percolation on the covering graph of~$G$.

Percolation is a minimal model which nonetheless captures aspects of the mechanism of interest --- that of water seeping through rock --- and provides qualitative and quantitative predictions. Since its introduction it has developed into an extremely rich subject, extensively studied by physicists and mathematicians. In this article we focus on one special facet of the theory, and so will certainly miss mentioning many interesting results. The interested reader is referred to the books \cite{\kesten,\grimmett,\brperc,\grimmettprob} for accounts of the mathematical theory. For an introduction to the extensive physics literature see \cite{\stauffer}.

\subsection{Critical percolation}

In percolation, a natural property to consider is the existence of an infinite open cluster: suppose $G=(V,E)$ is countably infinite and connected, and consider site percolation on $G$. For each site $x\in G$ let $C_x$ denote the open cluster containing $x$ (with $C_x = \emptyset$ if $x$ does not belong to an open cluster). We say that $x$ {\bf percolates} if $\abs{C_x}=\infty$; the physical interpretation is that $x$ ``gets wet.'' The {\bf (site) percolation probability} is defined to be
$$\thet_x(p)\equiv \thet_x^\site(p) \equiv\P_p^\site(|C_x|=\infty),$$
a non-decreasing function of $p$. Because $G$ is connected, for any $x,y\in G$, $\thet_x(p)$ and $\thet_y(p)$ must be both positive or both zero. Thus we can define a \textbf{critical (site) probability} for the graph $G$ as
$$p_c \equiv p_c^\site(G) \equiv \inf \{p : \thet_x(p) > 0\},$$
where $x$ is any site in $G$. The analogous quantities for bond percolation are denoted $\thet^\bond_x(p)$ and $p^\bond_c$.

We say {\bf percolation} occurs in the graph if $|C_x|=\infty$ for some $x\in G$. A common feature of many models in statistical physics is the {\bf phase transition}, where physical properties of the system undergo abrupt changes as a parameter passes a critical theshold. This is the case in percolation if $G$ has a translation invariance, for example, if $G$ is a lattice, or $G=G_0\times\Z$ for some base graph $G_0$: in this setting, the Kolmogorov zero-one law implies that the total number of infinite open clusters is a constant in $\set{0,1,\infty}$. The number is zero below $p_c$ and positive above $p_c$, so a sharp transition occurs at $p_c$. If $G$ is a lattice, $p_c\in(0,1)$ (e.g.\ by \eqref{eq:crit.prob}), and if an infinite open cluster exists then it is \emph{unique} (\cite{\akn}, see also \cite{\burtonkeane,\brperc}). For an example of a graph $G_0\times\Z$ with infinitely many infinite open clusters see \cite{\benjaminikozma}.

The effort to determine exact values for $p_c$ in various models, and to complete the picture of what happens at or near $p_c$, has led to a wealth of interesting probabilistic and combinatorial techniques. It was shown by Grimmett and Stacey \cite{\gsperc} that if $G=(V,E)$ is countably infinite and connected with maximum vertex degree $\De<\infty$, then
\beq \label{eq:crit.prob}
\f{1}{\De-1}
\le p^\bond_c
\le p^\site_c
\le 1-(1-p^\bond_c)^{\De-1}.
\eeq
Kesten proved that $p_c^\bond(\Z^2)=1/2$ \cite{\kestensq}. He also proved that $p_c^\bond(\Z^d) = (1+o(1))/(2d)$ as $d\to\infty$ \cite{\kestenhd}; better estimates were obtained by Hara and Slade (\cite{\hshighdim,\hslace,\hssaw}, see also \cite{\bkperchd}). For $\Lm=T$, Kesten showed $p_c = 1/2$ \cite{\kesten}. There are still famous open problems remaining in this field; for example, it is conjectured that
$\thet^\bond_x(p^\bond_c(\Z^d))=0$ for all $d\ge2$, but this has been proved only for $d=2$ \cite{\harrislbd,\kestensq} and $d\ge19$ \cite{\hshighdim} (see also \cite[Ch.~10-11]{\grimmett}). It is also known that $\thet_x(p_c(G))=0$ for both site and bond percolation on on any $G$ which is a Cayley graph of a finitely generated nonamenable group \cite{\blpsaop,\blpsgafa}.

The percolation phenomenon is also defined for a sequence of finite graphs $G_n=(V_n,E_n)$ with $|V_n|\to\infty$, where the analogue of the infinite component is the linear-sized or ``giant'' component. This has been studied most famously in the case of the Erd\dacute{o}s-R\'enyi random graphs ($G_n$ is the complete graph $K_n$), where it was found that around $p_n=1/n$ the size of the largest component has a ``double jump'' from $O(\log n)$ to $n^{2/3}$ to linear; see \cite{\bollobasrg} for historical background and references. Extremely detailed results are known about the structure of the open subgraph near criticality (see e.g.\ \cite{\dklp} and the references therein). There has also been work to obtain analogous results for other graph sequences: for example for increasing boxes in $\Z^d$ \cite{\bcksperc}, for vertex-transitive graphs (\cite{\nachmias,\bchssi,\bchssii,\bchssiii}), and for expander graphs \cite{\absperc}.

\subsection{Conformal invariance}

The purpose of this article is to explore the limiting structure of the percolation configuration in the {\bf scaling limit} of increasingly fine lattice approximations to a fixed continuous planar {\bf domain} $D$ (a nonempty proper open subset of $\C$ which is connected and simply connected). That is, for a planar lattice $\Lm$, we consider finite subgraphs $G^\de\subset \de\Lm$ which ``converge'' to $D$ as the mesh $\de$ decreases to zero (to be formally defined later).

It has been predicted by physicists that many classical models (percolation, Ising, FK) at criticality have scaling limits which are \emph{conformally invariant} and \emph{universal}. Recall that if $U\subseteq\C$ is an open set, $f:U\to\C$ is said to be {\bf conformal} if it is holomorphic (complex-differentiable) and injective.\footnote{Some authors more broadly define a conformal map to be any map which is holomorphic with non-vanishing derivative. Basic facts from complex analysis used in this article may be found in e.g.\ \cite{\ahlfors,\sscplx}.} Conformal maps are so called because they are rigid, in the sense that they behave locally as a rotation-dilation. The Riemann mapping theorem states that for any planar domains $D,D'$, there is a conformal bijection $\vph:D\to D'$. Roughly speaking, {\bf conformal invariance} means that \emph{the limiting (random) behavior of the model on $D'$ is the same (in law) as the image under $\vph$ of the limiting behavior on $D$}. {\bf Universality} means that the limiting behavior is not lattice-dependent.

The physics prediction seems quite surprising because the lattices themselves are certainly not conformally invariant. However, one might expect that lattice percolation on increasingly small scales becomes essentially ``locally determined,'' with the global lattice structure becoming insignificant. Conformal maps behave locally as rotation-dilations and so preserve local structure, so the conformal invariance property means heuristically that the scaling limit of lattice percolation is locally scale-invariant and rotation-invariant. Scale-invariance follows essentially by definition of the scaling limit, and rotation-invariance can be hoped for based on the symmetry and homogeneity of the lattice.

In fact there is a classical example of a conformally invariant scaling limit: the planar Brownian motion, which we discuss in \S\ref{ch:bm}. This is the process $\bm_t = \bm_t^1 + i \bm_t^2$ where the $\bm_t^i$ are independent standard Brownian motions in $\R$. By Donsker's invariance principle, the universal scaling limit of planar random walks with finite variance is $\mu t + A\bm_t$ where $\mu\in\R^2$, $A\in\R^{2\times 2}$ (see e.g.\ \cite{\ks}). Using the idea of ``local determination'' described above, L\'evy deduced in the 1930s that planar Brownian motion is conformally invariant, up to a time reparametrization \cite{\levy}.  L\'evy's proof, however, was not completely rigorous, and most modern proofs of this result use the theory of stochastic integrals (developed by It\=o in the 1950s). Underlying this conformal invariance are some important connections between Brownian motion and harmonic functions, and we will make precise some of these connections which were first observed by Kakutani \cite{\kakutani}.
\vfill

\subsection{The percolation scaling limit}

This article describes two major breakthroughs, due to Schramm \cite{\schramm} and Smirnov \cite{\smirnovperc,\smirnovpercarxiv}, which gave the rigorous identification in the early 2000s of the scaling limit of critical percolation on the triangular lattice $T$.

\subsubsection{Smirnov's theorem on crossing probabilities}

Langlands, Pouliot, and Saint-Aubin \cite{\lpsa}, based on experimental observations and after conversations with Aizenman, conjectured that critical lattice percolation has a conformally invariant scaling limit. Some mathematical evidence was provided by Benjamini and Schramm \cite{\bsvoronoi}, who proved that a different but related model, Voronoi percolation, is invariant with respect to a conformal change of metric.

Using non-rigorous methods of conformal field theory (CFT), physicists were able to give very precise predictions about various quantities of interest in planar critical percolation. Cardy \cite{\cardy,\cardylec} notably derived an exact formula for the (hypothetical) limiting probability of an open crossing between disjoint boundary arcs of a planar domain. Carleson made the important observation that this formula has a remarkably simple form when the domain is an equilateral triangle (see \S\ref{sec:perc.completing}). However, for years mathematicians were unable to rigorously justify the CFT methods used in Cardy's derivation.

In 2001, Smirnov \cite{\smirnovpercarxiv, \smirnovperc} proved that for site percolation on the triangular lattice $T$, the limiting crossing probability exists and is conformally invariant, satisfying Cardy's formula. The purpose of \S\ref{ch:perc} is to give an exposition of this result. The proof is based on the discovery of ``preharmonic'' functions which encode the crossing probability and converge in the scaling limit to conformal invariants of the domain.\footnote{We use the term preharmonic rather than ``discrete harmonic,'' which also appears in the literature, to avoid confusion with the classical meaning of discrete harmonic (a function whose value at any vertex is the average of the neighboring values) which is not necessarily what is meant here.} Although the percolation scaling limit is believed to be universal, special symmetries of the triangular lattice play a crucial role in Smirnov's proof, and the result has not been extended to other lattices. For recent work on this question see \cite{\beffarauniv,\chayeslei,\bclperci,\bclpercii}.

The exposition of Smirnov's theorem in \S\ref{ch:perc} is partly based on the one in \cite[Ch.~7]{\brperc}. For different perspectives (in addition to the original works of Smirnov) see \cite{\beffara,\grimmettprob,\smirnovicmplisbon}.

The general principle of preharmonicity and preholomorphicity has been further developed by Chelkak, Hongler, Kemppainen, and Smirnov in establishing conformal invariance in the scaling limit of the critical Ising and FK models \cite{\smirnovtowards,\smirnovrci,\smirnovrcii,\ksrciii,\chelkaksmirnovuniv,\chelkaksmirnovising,\honglersmirnov}. Discrete complex analysis appears also in the work of Duminil-Copin and Smirnov \cite{\dcssaw} determining the connective constant of the hexagonal lattice, which makes substantial progress towards establishing a conformally invariant scaling limit for the self-avoiding walk (SAW). For a more general discussion and references see \cite{\smirnovicmhyderabad,\dcsmirnovsurvey}.

\subsubsection{Schramm-Loewner evolutions}

While the notion of a limiting crossing probability is easy to define (though it may not exist), it is not immediately clear how to formally define the ``limiting percolation configuration.'' This notion is discussed in the work of Aizenman \cite{\aizenmannumber,\aizenmansclim}, and the 1999 work of Aizenman and Burchard \cite{\abholder} shows how to obtain \emph{subsequential} scaling limits of the percolation configuration. At the time, no direct construction for the limiting object --- that is, a construction not involving limits of discrete systems --- was available.

Such a construction was discovered in 1999-2000 when, in the course of studying the scaling limit of the loop-erased random walk (LERW), Schramm gave an explicit mathematical description of a one-parameter family of conformally invariant random curves, now called the \emph{Schramm-Loewner evolutions} ($\sle$). These curves are characterized by simple axioms which identify them as essentially the universal candidate for the scaling limits of macroscopic interfaces in planar models.

The theory of $\sle$ contains some very beautiful mathematics and is one of the major developments of probability theory within the past decade, and \S\ref{ch:sle} aims to give an accessible introduction. Here is a brief preview, glossing over all technical details: the $\sle$ are a one-parameter family of self-avoiding\footnote{A formal definition appears in \S\ref{sec:lde.sle}. Self-avoiding curves are \emph{not} necessarily simple; indeed, the scaling limit of the percolation interface will have many double points. Informally, a self-avoiding curve is a curve without transversal self-crossings.} random planar curves $\gam \equiv \gam(D;a,b)$ traveling from $a$ to $b$ in domain $D$, where $a$ is in the boundary $\pd D$ and $b$ is either in $D$ ({\bf radial}) or elsewhere on $\pd D$ ({\bf chordal}). The curves satisfy two axioms:
\begin{enumerate}[(1)]
\item {\bf Conformal invariance}: if $\vph$ is a conformal map defined on domain $D$, $\vph\gam(D;a,b)$ has the same law as $\gam(\vph D;\vph a,\vph b)$.

\item {\bf Domain Markov property}: conditioned on $\gam[0,t]$, the remaining curve has the same law as $\gam(D(t),\gam(t),b)$ where $D(t)$ is the unique connected component of $D\setminus \gam[0,t]$ whose closure contains $b$ (the ``slit domain'').
\end{enumerate}
Conformal invariance is expected by physicists as already mentioned, and typically the domain Markov property holds in the discrete setting and is believed to pass to the scaling limit.

Schramm realized that these two properties essentially fully determine the distribution of the curve. His discovery is based on the \emph{Loewner differential equation} (LDE), which describes the evolution of a self-avoiding curve $\gam(D;a,b)$ through the evolution of the corresponding conformal mappings $g_t : D(t) \to D$ (the ``slit mappings''). If the domain is the upper-half plane $\H$ with marked boundary points $0,\infty$, we will see that if the maps $g_t$ are normalized to ``behave like the identity'' near $\infty$, then under a suitable time parametrization they satisfy the \emph{chordal LDE}
$$\dot{g}_t(z)=\frac{2}{g_t(z)-u_t},\quad g_0(z)=z \quad\forall z\in\H,$$
where $u_t\equiv g_t(\gam(t))$ is a continuous real-valued process, called the \emph{driving function}. (The radial version of the LDE was developed by Charles Loewner in 1923 and used by him to prove a case of the Bieberbach conjecture; see \cite{\ahlforsconf}.) This equation is remarkable because it encodes the planar curve $\gam$ in the one-dimensional process $u_t$. Given $u_t$, one can recover the original curve $\gam$ by solving the ordinary differential equation above for each $z$ up to time $t$, and setting $\gam(t)=g_t^{-1}(u_t)$.

\begin{figure}
\subfloat[Original $\sle$ curve $\gam=\gam_{D,a}$]{\label{fig:domain.mp}\includegraphics[width=0.47\textwidth]{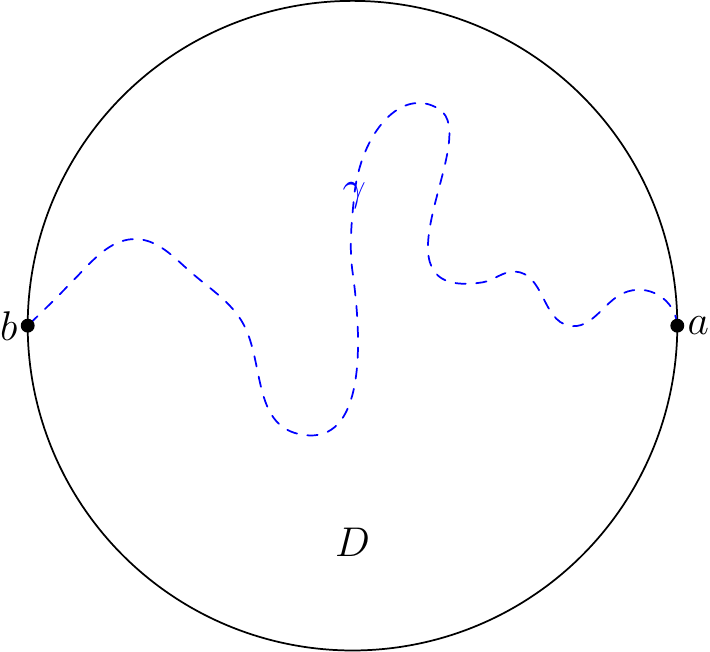}}
\qquad
\subfloat[{$\gam[t,\infty)$ conditioned on $\gam[0,t]$}]{\label{fig:domain.mp.cond}\includegraphics[width=0.47\textwidth]{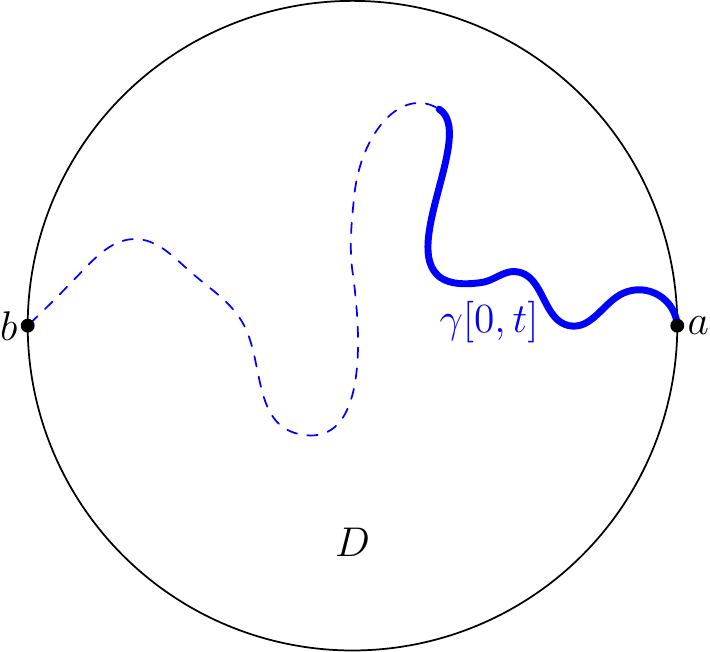}}
\\
\subfloat[{$\gam_{D(t),\gam(t)}$ where $D(t)=D\setminus\gam[0,t]$}]{\label{fig:domain.mp.slit}\includegraphics[width=0.47\textwidth]{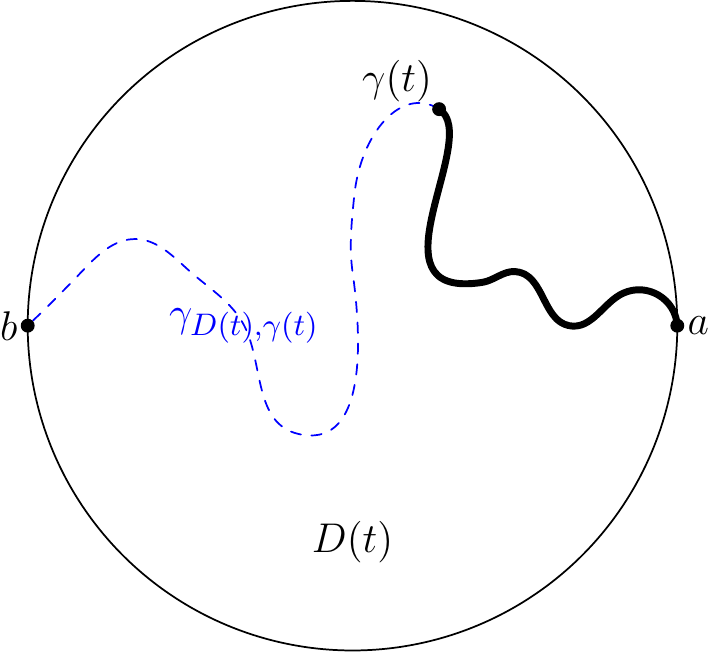}}
\qquad
\subfloat[{$\gam_{D(t),\gam(t)}$ mapped to original domain}]{\label{fig:domain.mp.mapped}\includegraphics[width=0.47\textwidth]{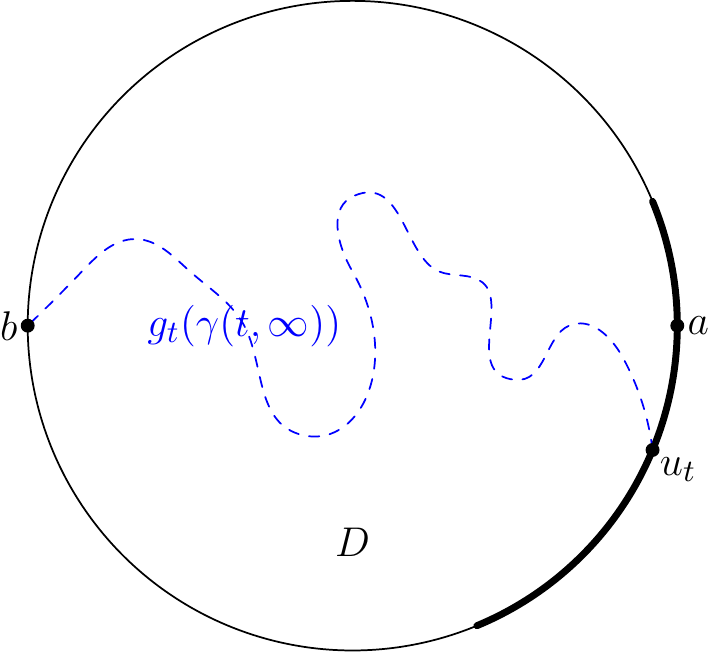}}
\caption{Domain Markov property}
\end{figure}

The $\sle$ axioms imply that if we condition on $\gam[0,t]$, the image of the remaining curve under $g_t$ is distributed as $u_t+\tilde\gam$ where $\tilde\gam$ is an independent realization of $\gam$. Figs.~\ref{fig:domain.mp} through~\ref{fig:domain.mp.mapped} illustrate this idea. As Schramm noted, this implies that \emph{$u_t$ must be a Brownian motion}, $u_t = \mu t + \sqrt{\ka}\bm_t$. Further, under symmetry assumptions $u$ will have zero drift. The Schramm-Loewner evolutions $\sle_\ka$ are the curves recovered from the LDE with driving function $u_t = \sqrt{\ka}\bm_t$.

In fact, it will require work to show that ``$u_t=g_t(\gam(t))$'' is even well-defined, and it is not trivial to find conditions under which a true curve is recovered from the LDE. A detailed study of the geometric properties of deterministic Loewner evolutions may be found in \cite{\mrloewner}. That the $\sle_\ka$ processes are true curves is a difficult theorem, proved for $\ka=8$ in \cite{\lswlerw} and for $\ka\ne8$ in \cite{\rohdeschramm} (see \cite[Ch.~7]{\lawler}). We will not make use of this fact.

In \S\ref{ch:sle} we follow for the most part Lawler's book \cite{\lawler}, which contains far more information than can be covered here. See also the lecture notes \cite{\wernerstflour,\lawlerparkcity}. In the decade since Schramm's original paper there have been many works investigating properties of $\sle$, e.g.\ the Hausdorff dimension \cite{\beffarahdim,\beffaradim}. Connections to (planar) Brownian motion are discussed in \cite{\lswintersecI,\lswintersecII,\lswintersecIII,\lswrestr}.  There has been work on characterizing more measures with conformal invariance properties, e.g.\ the restriction measures of \cite{\lswrestr}. Sheffield and Werner have given a characterization of conformally invariant loop configurations, the {\bf conformal loop ensemble} ($\cle$) \cite{\swcle}.

\subsubsection{\texorpdfstring{Percolation exploration path and convergence to $\sle_6$}{Percolation exploration path and convergence to SLE(6)}}

Schramm conjectured that interfaces in the percolation model converge to forms of $\sle_6$, and the conformal invariance of crossing probabilities was the key to proving this result. In his work \cite{\smirnovperc,\smirnovpercarxiv}, Smirnov outlined a proof for the conformal invariance of the \emph{full percolation configuration} (as a collection of nested curves). His outline was later expanded into detailed proofs in work by Camia and Newman \cite{\cmnw,\cmnwfull}.

The proof for the full configuration is beyond the scope of this article, and instead we focus on a single macroscopic interface. The percolation \emph{exploration path} is defined roughly as follows (the formal definition appears in \S\ref{sec:perc.end}): fix two points $a,b$ on the boundary of a simply connected domain. Fix all the hexagons on the counterclockwise arc $\ol{ab}$ to be closed, and all those on $\ol{ba}$ to be open. The exploration path is the interface curve which separates the closed cluster touching $\ol{ab}$ from the open cluster touching $\ol{ba}$. Fig.~\ref{fig:explore} shows part of an exploration path traveling in $\H$ between 0 and $\infty$. The purpose of \S\ref{ch:scaling}~and~\S\ref{ch:explore} is to prove that the exploration path converges to chordal $\sle_6$.

\begin{figure}
\centering
\includegraphics[width=\linewidth]{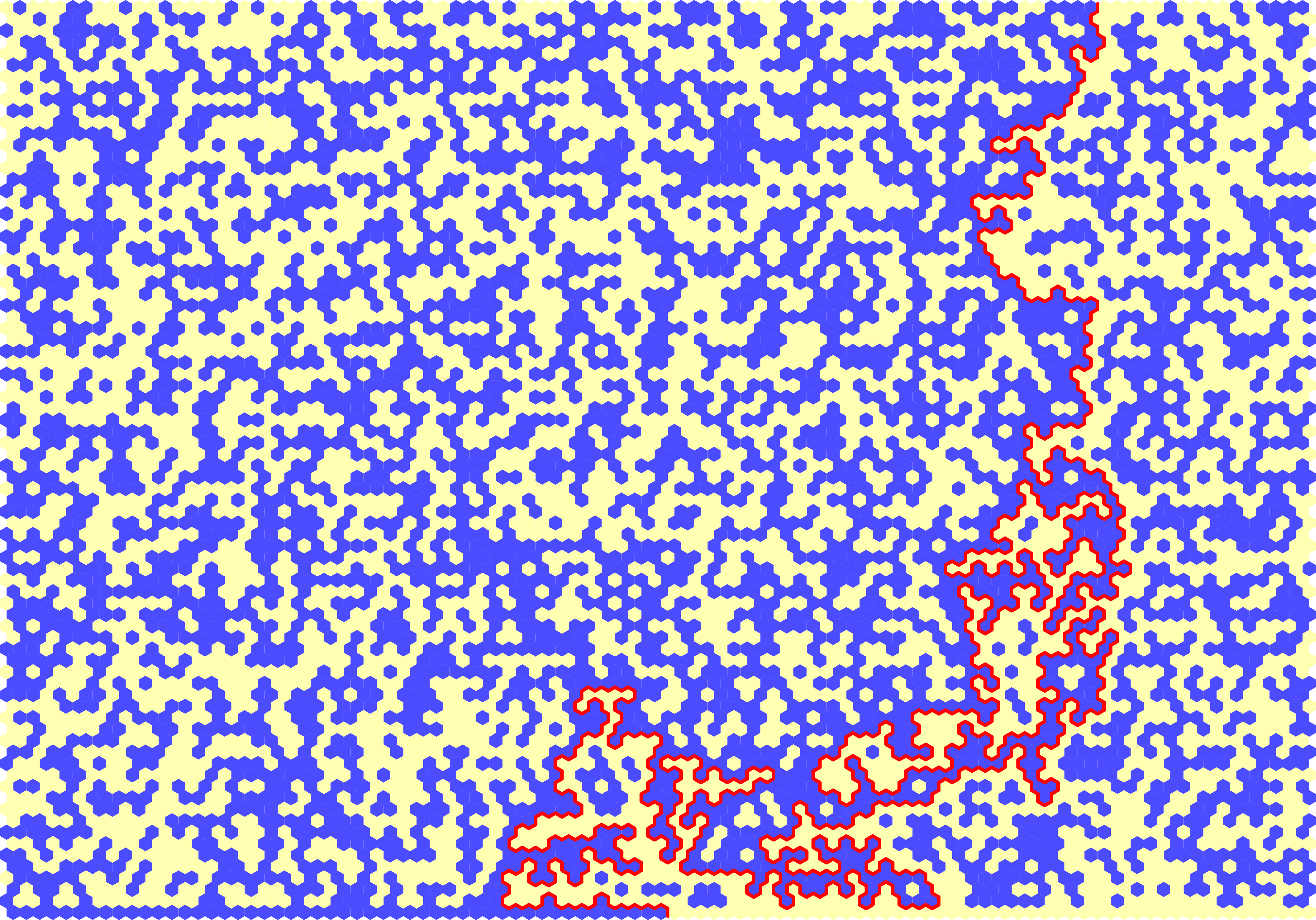}
\caption{Percolation exploration path}
\label{fig:explore}
\vspace*{-12pt}
\end{figure}

\subsubsection{Outline of remaining sections}

The organization of the remainder of this article is as follows:
\begin{itemize}
\item In \S\ref{ch:bm} we give a proof of the conformal invariance of Brownian motion, and discuss the connection to harmonic functions first developed by Kakutani \cite{\kakutani}. The material presented in this section is elementary in nature and can be skipped without loss by readers already familiar with the subject.

\item In \S\ref{ch:perc} we present Smirnov's theorem \cite{\smirnovperc,\smirnovpercarxiv} on the conformal invariance of crossing probabilities in critical percolation on the triangular lattice $T$.

\item \S\ref{ch:sle} is an introduction to the Loewner differential equation and $\sle$. \S\ref{ssec:sle.six} states a characterization of $\sle_6$ (following Camia and Newman \cite{\cmnw}) which will be used in identifying the scaling limit of the exploration path.

\item \S\ref{ch:scaling} presents a result of Aizenman and Burchard \cite{\abholder} which uses {\it a priori} estimates (in our case, percolation crossing exponents) to deduce the existence of \emph{subsequential} weak limits for the percolation exploration path.

\item In \S\ref{ch:explore} we show that all subsequential limits of the exploration path are chordal $\sle_6$. Our exposition follows the work of Binder, Chayes and Lei \cite{\bclperci,\bclpercii} and of Camia and Newman \cite{\cmnw}.

\item \S\ref{ch:end} concludes with some additional references and open problems concerning other planar models.
\end{itemize}
For more accounts on conformal invariance in percolation and convergence to $\sle_6$, see the lecture notes of Werner \cite{\wernerparkcity} and of Beffara and Duminil-Copin \cite{\beffaradcopin}. For a broad overview of conformally invariant scaling limits of planar models see the survey of Schramm \cite{\schrammmadrid}.

\subsection*{Acknowledgements}

This survey was originally prepared as an undergraduate paper in mathematics at Harvard University under the guidance of Yum-Tong Siu and Wilfried Schmid, and I thank them both for teaching me about complex analysis, and for their generosity with their time and advice.

I am very grateful to Scott Sheffield for teaching me about percolation and $\sle$, for answering countless questions and suggesting useful references. I would like to thank Scott Sheffield and Geoffrey Grimmett for carefully reading drafts of this survey and making valuable suggestions for improvement which I hope are reflected here.

I thank Amir Dembo, Jian Ding, Asaf Nachmias, David Wilson, and an anonymous referee for many helpful comments and corrections on recent drafts of this survey. I thank Michael Aizenman, Almut Burchard, Federico Camia, Curtis McMullen, Charles Newman, and Horng-Tzer Yau for helpful communications during the course of writing.

\vspace*{-2pt}
\section{Brownian motion}
\label{ch:bm}
\vspace*{-2pt}

Some notations: let $\ball[z]{r}$ denote the open ball of radius $r$ centered at $z\in\R^d$, and $\circl[z]{r}\equiv\pd\ball[z]{r}$. Let $\ball{r}\equiv\ball[0]{r}$, $\circl{r}\equiv\pd\ball{r}$ denote the versions of these which are centered at the origin.

If $(\Om,\cF,\P)$ is a probability triple and $(\cF_t)_{t\ge0}$ is an increasing filtration in $\cF$, we call $(\Om,\cF,(\cF_t)_{t\ge0},\P)$ a filtered probability space. For $X\in\cF_0$, an {\bf $(\cF_t)$-Brownian motion} started at $X$ is a stochastic process $(\bm_t)_{t\ge0}$ defined on $(\Om,\cF,\P)$ which is $(\cF_t)$-adapted and satisfies
\begin{enumerate}
\item $\bm_0=X$ a.s.;
\item $\bm$ is a.s.\ continuous;
\item for $s<t$, $\bm_t-\bm_s\sim\dN(0,t-s)$ and is independent of $\cF_t$.
\end{enumerate}
A {\bf Brownian motion in $\R^d$} is simply a $d$-dimensional vector of independent Brownian motions $(\bm_t^1,\ldots,\bm_t^d)$. For $d=2$ we always identify $\R^2\cong\C$ and write $\bm_t\equiv\bm_t^1+i\bm_t^2$ ($t \ge 0$), referred to as planar or complex Brownian motion. An approximation of a sample path is shown in Fig.~\ref{fig:bm} (the path becomes lighter over time). For $z\in\C$ let $\P_z$ denote the probability measure for Brownian motion started at $z$, and write $\E_z$ for expectation with respect to $\P_z$. For more on Brownian motion see e.g.\ \cite{\ks,\mpbm}.

\begin{figure}[t!]
\centering
\includegraphics[width=\textwidth]{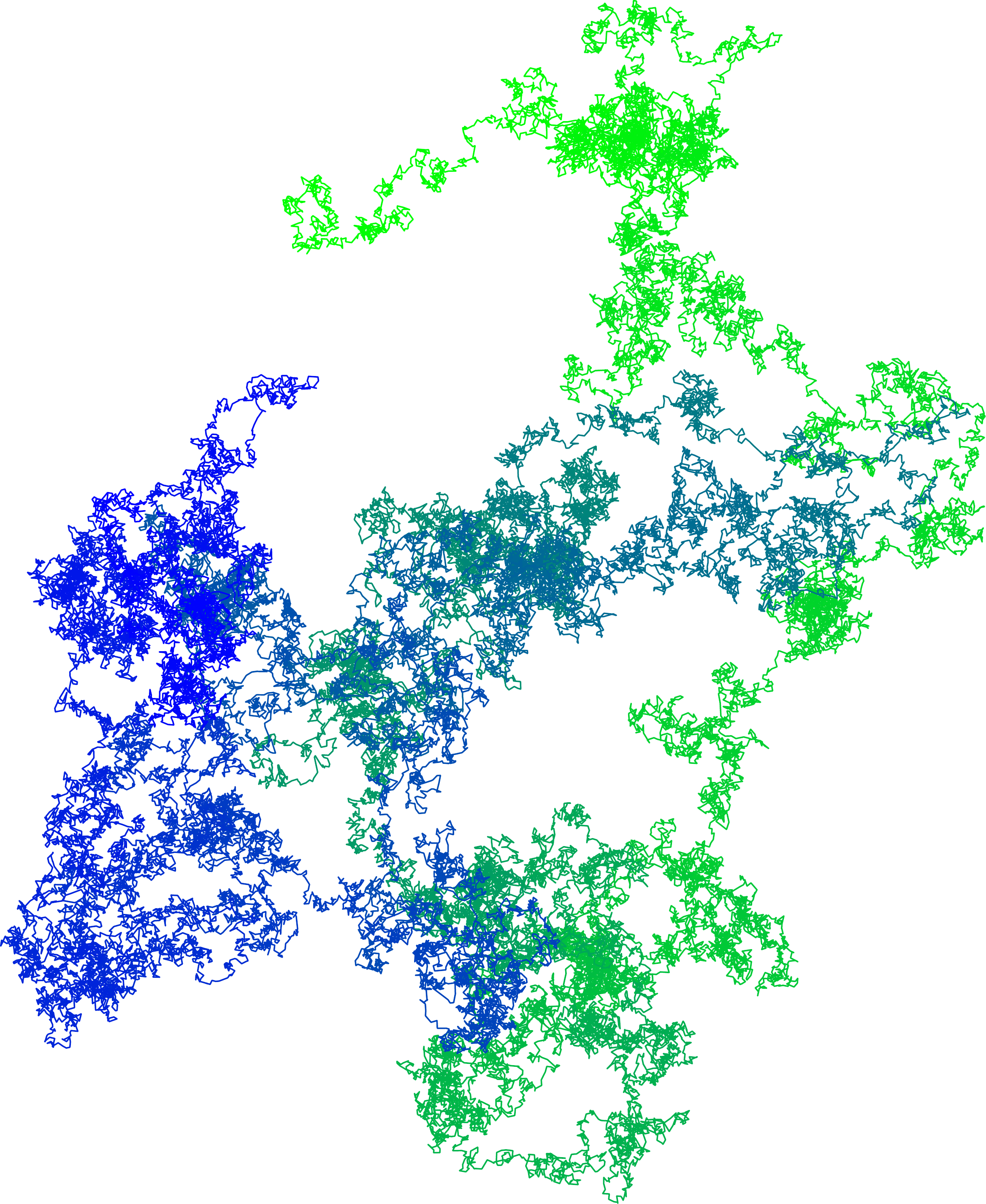}
\caption{Approximate sample path of Brownian motion}
\label{fig:bm}
\end{figure}

\vspace*{-2pt}
\subsection{Conformal invariance of planar Brownian motion}
\vspace*{-2pt}

The result below tells us that the holomorphic image of a Brownian motion is again a Brownian motion up to a (random) time change. For $D$ a complex domain\footnote{Recall that by a complex domain we always mean a nonempty proper open subset of $\C$ which is connected and simply connected.} and $\bm$ a Brownian motion started at $z_0\in D$, let
$$\tau_D\equiv
\inf\set{t\ge0:\bm_t\notin D}$$
denote the {\bf exit time} of $D$.

\begin{thm} \label{thm:bm.ci}
Let $\vph$ be a holomorphic map defined on the complex domain $D$, and for $\bm$ a Brownian motion started at $z_0\in D$ let $Y_t =\vph(\bm_t)$ ($t<\tau_D$). If
$$Y_t\equiv \wt\bm_{\si(t)},\quad \si(t) \equiv \int_0^t |f'(\bm_s)|^2 \d{s}$$
then $\wt\bm_t$ is a Brownian motion started at $\vph(z_0)$ run up to $\tau_{\vph(D)}$.
\end{thm}
\eject

Here is L\'evy's heuristic argument \cite{\levy}: translations and rotations of $\C$ clearly map Brownian motions to Brownian motions, and we have the Brownian scaling $\wt\bm_t=cW_{t/c^2}$ ($c>0$) for $\wt\bm$ another Brownian motion of the same speed as $W$. The stochastic process is determined by its behavior ``locally'' at each point in time, and near time $t$ we have
$$\vph(\bm_{t+h}) - \vph(\bm_t)
\approx \lm_t (\bm_{t+h} - \bm_t), \quad
\lm_t\equiv\vph'(\bm_t).$$
By rotation-invariance and the Brownian scaling,
$\lm_t (\bm_{t+h} - \bm_t)=\wt\bm_{\abs{\lm_t}^2 h}$, so $\vph(\bm)$ at time $t$ behaves like a standard Brownian motion ``sped up'' by a factor of $|\lm_t|^2$.

Below we present the modern proof of this result, which uses stochastic calculus.\footnote{Stochastic calculus is not heavily used in this article, so the reader not familiar with the subject will not lose much by simply skipping over the places where it appears.} In preparation we first recall It\=o's formula, L\'evy's characterization of Brownian motion, and the Dubins-Schwarz theorem (proofs may be found in e.g.\ \cite{\ks}). For two continuous semimartingales $X,X'$ we use the notation $[X,X']$ for their covariation process, and we write $[X]\equiv[X,X]$ for the quadratic variation process of $X$.

\begin{thm}[It\=o's formula]
Let $X$ be a continuous semimartingale: $X_t=X_0 + M_t + A_t$ where $M$ is a local martingale and $A$ is a finite-variation process. Then $f(X_t)$ is again a continuous semimartingale with
$$f(X_t) - f(X_0)
= \sum_{j=1}^d \int_0^t \pd_j f(X_s) \d X^j_s
+ \f{1}{2} \sum_{j,k=1}^d \int_0^t \pd_j \pd_k f(X_s) \d[X^j,X^k]_s.$$
In differential notation
$$df(X_t)
= \sum_{j=1}^d \pd_j f(X_s) \d X^j_s
+ \f{1}{2} \sum_{j,k=1}^d \pd_j \pd_k f(X_s) \d[X^j,X^k]_s.$$
\end{thm}

\begin{thm}[L\'evy's characterization of Brownian motion]
Let $(X_t)_{t\ge0}$ be a continuous adapted process in $\R^d$ defined on the filtered probability space $(\Om,\cF,(\cF_t)_{t\ge0},\P)$ such that
\begin{enumerate}[(i)]
\item $X^j_t-X^j_0$ is a local $(\cF_t)$-martingale for $1\le j\le d$;
\item the pairwise quadratic variations are $[X^j,X^k]_t = \de_{kj} t$ for $1\le k,j\le d$.
\end{enumerate}
Then $X$ is an $(\cF_t)$-Brownian motion in $\R^d$.
\end{thm}

\begin{thm}[Dubins-Schwarz theorem] Let $M$ be a continuous local martingale defined on the filtered probability space $(\Om,\cF,(\cF_t)_{t\ge0},\P)$ with $[M]_t\to\infty$ a.s. If
$$\tau(s) \equiv \inf\set{t\ge0:[M]_t>s},\quad
\cG_s\equiv \cF_{\tau(s)},$$
then $\bm_s\equiv M_{\tau(s)}$ is a $(\cG_t)$-Brownian motion, and $M_t=\bm_{[M]_t}$.
\end{thm}

\begin{proof}[Proof of Thm.~\ref{thm:bm.ci}]
Let $u=\real\vph$ and $v=\imag\vph$; we regard $u,v$ as functions on $\R^2$ and write $\vph(x+iy)=u(x,y)+i v(x,y)$. Writing $\bm_t=\bm^1_t+i\bm^2_t$ and applying It\=o's formula (using the indepence of the two components) gives
\begin{align*}
d u(\bm_t) &= \pd_x u(\bm_t) \d\bm^1_t+ \pd_y u (\bm_t) \d\bm^2_t
    + \f{1}{2} \De u(\bm_t) \d{t}, \\
d v(\bm_t) &= \pd_x v(\bm_t) \d\bm^1_t+ \pd_y v (\bm_t) \d\bm^2_t
    + \f{1}{2}\De u(\bm_t) \d{t}.
\end{align*}
By the Cauchy-Riemann equations,
$$d[u(\bm)]_t = \lp [\pd_x u(\bm_t)]^2 + [\pd_y u(\bm_t)]^2 \rp \d t
= |f'(\bm_t)|^2 \d t
= d[v(\bm)]_t$$
and
$$d[u(\bm),v(\bm)]_t
= \lb \pd_x u(\bm_t) \pd_x v(\bm_t) + \pd_y u(\bm_t) \pd_y v(\bm_t) \rb dt
= 0.$$
Further $u,v$ are harmonic (again by the Cauchy-Riemann equations) so $u(\bm_t)$ and $v(\bm_t)$ are local martingales, hence time-changed Brownian motions by the Dubins-Schwarz theorem:
$$u(\bm_t) = \wt\bm^1_{\si(t)},\quad
u(\bm_t) = \wt\bm^2_{\si(t)}$$
for Brownian motions $\wt\bm^1,\wt\bm^2$. (One needs to verify that $\si(t)\to\infty$; this follows from holomorphicity of $\vph$ and the neighborhood recurrence of planar Brownian motion.) Finally, we leave the reader to check that $[\wt\bm^1,\wt\bm^2]=0$; the result then follows from L\'evy's characterization of Brownian motion.
\end{proof}

\subsection{Harmonic functions and the Dirichlet problem}
\label{sec:bm.harm}

Given a domain $D$ and a bounded continuous function $f : \pd D\to\R$, we say that $u\in C^2(D) \cap C(\clos D)$ solves the {\bf Dirichlet problem} on $D$ with boundary data $f$ if $u$ is harmonic ($\De u=0$) in $D$ and agrees with $f$ on $\pd D$. We now demonstrate the connection to Brownian motion as first noted by Kakutani \cite{\kakutani}. Recall that $\tau_D$ denotes the exit time of a Brownian motion from $D$, and let $\tau_D^+ \equiv \inf\set{t>0 : \bm_t\notin D}$. A domain $D$ is said to be {\bf regular} if for all $z\in\pd D$, $\tau_D^+=0$, $\P_z$-a.s.

\begin{ppn}
\label{ppn:bm.dirichlet}
Let $D$ be a bounded regular domain, and let $f : \pd D \ra \R$ continuous. Define $u : \clos{D} \ra \R$ by
$$u(z) = \left\{ \begin{array}{cl}
    f(z) & \text{ if } z \in \pd D. \\
    \E_z [f(\bm_{\tau_D})] & \text{ if } z \in D.
    \end{array} \right.$$
Then $u$ is a bounded solution to the Dirichlet problem on $D$ with boundary data $f$, and it is the unique such solution.

\begin{proof}[Proof (sketch)]
Since $f$ is bounded, $u$ is certainly bounded. Uniqueness is easy: let $v$ be another such solution. Then $v(\bm_t)$ is a bounded local martingale, hence a uniformly integrable martingale, so by the optional stopping theorem
$$v(z)
= v(\bm_0)
= \E_z[v(\bm_{\tau_D})]
= \E_z[f(\bm_{\tau_D})]
= u(z).$$
To show $u$ is harmonic we check the local mean-value property, that for each $z\in D$ there exists $r_0>0$ such that
$$u(z) = \frac{1}{2\pi} \int_0^{2\pi} u(z + re^{i\thet}) \ d\thet,$$
for $0<r\le r_0$ (see e.g.\ \cite{\liebloss}). Let $\bm_t$ be a Brownian motion started at $z$, and let $\si=\tau_{\ball[z]{r}}$. By the rotational symmetry of Brownian motion, $\bm_\si$ is uniformly distributed on $\pd \ball[z]{r}$, so the right-hand side of the above is
\begin{align*}
\E_z [u(\bm_\si)]
&= \E_z \left[ \E_{\bm_\si} [f(\bm_{\tau_D})] \right]
    \quad\text{by definition of $u$} \\
&= \E_z \left[ \E_z [f(\bm_{\tau_D}) | \cF_\si] \right]
    \quad\text{by strong Markov property} \\
&= \E_z [f(\bm_{\tau_D})]
    \quad\text{by iterated expectations} \\
&= u(z).
\end{align*}
This shows that $u$ is harmonic. Continuity of $u$ on $\clos D$ requires the regularity assumption and we omit the proof here.
\end{proof}
\end{ppn}

\begin{rmk}
For some simple domains $U$ a {\bf Poisson integral formula} gives an explicit mapping $P_U$ from a (bounded continuous) function $f$ defined on $\pd U$ to the solution of the Dirichlet problem on $U$ with boundary values $f$. The formulas for the disc and the upper half-plane are well known and will be used in deriving the Loewner differential equation, so we recall them here:
\begin{align}
\label{eq:pois.d}
&(P_\D f)(re^{i\thet})
= \frac{1}{2\pi} \int_0^{2\pi} P_r(\thet-\vph) f(e^{i\vph}) \ d\vph,
\quad
P_r(\thet)=\real\left(\frac{1+re^{i\thet}}{1-re^{i\thet}}\right),\\
\label{eq:pois.h}
&(P_\H f)(x+iy) = \frac{1}{\pi} \int_\R Q_y(t-x) f(t) \ dt, \quad
    Q_y(x)=\frac{y}{x^2+y^2}.
\end{align}
For proofs see e.g.\ \cite{\ssf,\sscplx}.
\end{rmk}

Propn.~\ref{ppn:bm.dirichlet} also gives us a weaker version of Thm.~\ref{thm:bm.ci}, namely, that the \emph{hitting} distribution of Brownian motion is conformally invariant. Indeed, let $D,D'$ be regular domains (say with smooth boundaries), and let $z$ be a point in $D$ and $A$ an arc on $\pd D$. Let $z'=\vph(z)$, $A'=\vph(A)$, and $\bm'_t=\vph(\bm_t)$. Then $u(z)\equiv\P_z(\bm_{\tau_D}\in A)$ solves the Dirichlet problem on $D$ with boundary conditions $f=\I_A$, while $u'(z)\equiv\P_{z'}(\bm_{\tau_{D'}}\in A')$ solves the Dirichlet problem on $D'$ with boundary conditions $f'=\I_{A'}$.\footnote{The indicator functions are discontinuous, but since we assumed $\pd D$ to be smooth we may easily approximate $\I_A$ by continuous functions.} But
$$u\circ\vph^{-1}$$
is also a solution, so by uniqueness we have $u=u'\circ\vph$. Since $A$ was arbitrary the conformal invariance of the hitting distribution follows.

In fact, here is a way to deduce the full conformal invariance of the Brownian path (modulo time reparametrization) from this seemingly weaker result. For each $\ep>0$, we approximate the Brownian path by the piecewise linear curve $\bm_t^{(\ep)}$ joining the points $z=z_1,z_2,\ldots$, where $z_{k+1}$ is the first point where the Brownian motion starting from $z_k$ hits $\ball[z_k]{\ep}$. To be precise, we will define a sequence of stopping times $0=\tau_0,\tau_1,\ldots$ by recursively setting $z_k=\bm_{\tau_k}$ and $\tau_{k+1}=\tau_{\ball[z_k]{\ep}}$. We then make $\bm_t^{(\ep)}$ into a continuous stochastic process by linear interpolation. It is clear that as $\ep\to0$, $\bm_t^{(\ep)}$ converges a.s.\ to $\bm_t$ in the topology of uniform convergence. Now, if $\vph$ is a conformal map, then by the above $\vph(\bm_{\tau_{k+1}}^{(\ep)})$ has the same distribution as the point where a Brownian motion started at $\vph(z_k)$ first exits the conformal ball $\vph(\ball[z_k]{\ep})$. It follows that $\vph(\bm_t^{(\ep)})$ converges a.s.\ as $\ep\to0$ to a process which is a time reparametrization of Brownian motion.

The proof of Smirnov's theorem for crossing probabilities, in the next section, is to some extent motivated by these simple observations. In particular, conformal invariance of the crossing probabilities will follow naturally from expressing the probabilities in terms of a harmonic function solving some form of Dirichlet problem. The idea of constructing polygonal approximations to a random path will also reappear, in \S\ref{ch:explore}, when we use a modification of this construction to show that the scaling limit of percolation agrees with chordal $\sle_6$.

\section{Percolation and Smirnov's theorem on crossing probabilities}
\label{ch:perc}

In this section we present Smirnov's celebrated theorem on the conformal invariance of crossing probabilities in critical percolation on the triangular lattice $T$ \cite{\smirnovperc, \smirnovpercarxiv}. Smirnov's key insight was the discovery of a ``preharmonic'' function whose evaluation at a certain point gives the crossing probability. As $\de\to0$, these functions converge to a true harmonic function solving a Dirichlet-type problem on $D$, and the theorem follows because the solution to the Dirichlet problem is a conformal invariant.

The first section below gives a formal statement of Smirnov's crossing probabilities theorem.

\subsection{Statement of Smirnov's theorem}
\label{sec:perc.planar}

Let $\Lm=(V,E)$ be a planar lattice. Write $u\sim v$ if $(u,v)\in E$. Site percolation on $\Lm$ may be visualized as {\bf face percolation} on the dual lattice $\Lm^*$. We will use {\bf blue} and {\bf yellow} in the place of open and closed respectively, particularly to avoid confusion with the topological meanings of those words.

\begin{figure}
\subfloat[$T$ with dual $H$]{\includegraphics[height=0.42\textwidth]{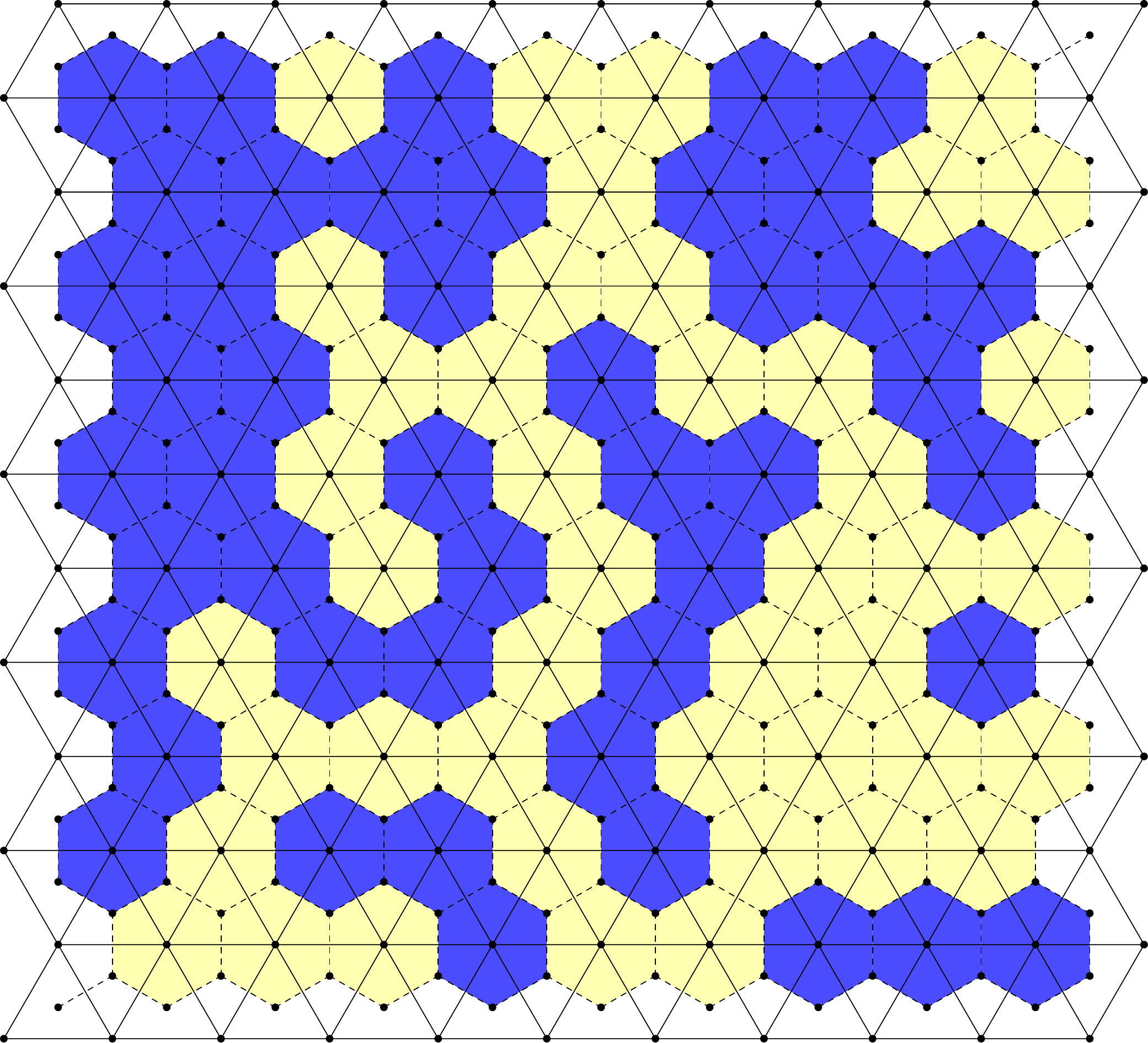}}
\qquad
\subfloat[Face percolation on $H$]{\includegraphics[height=0.42\textwidth]{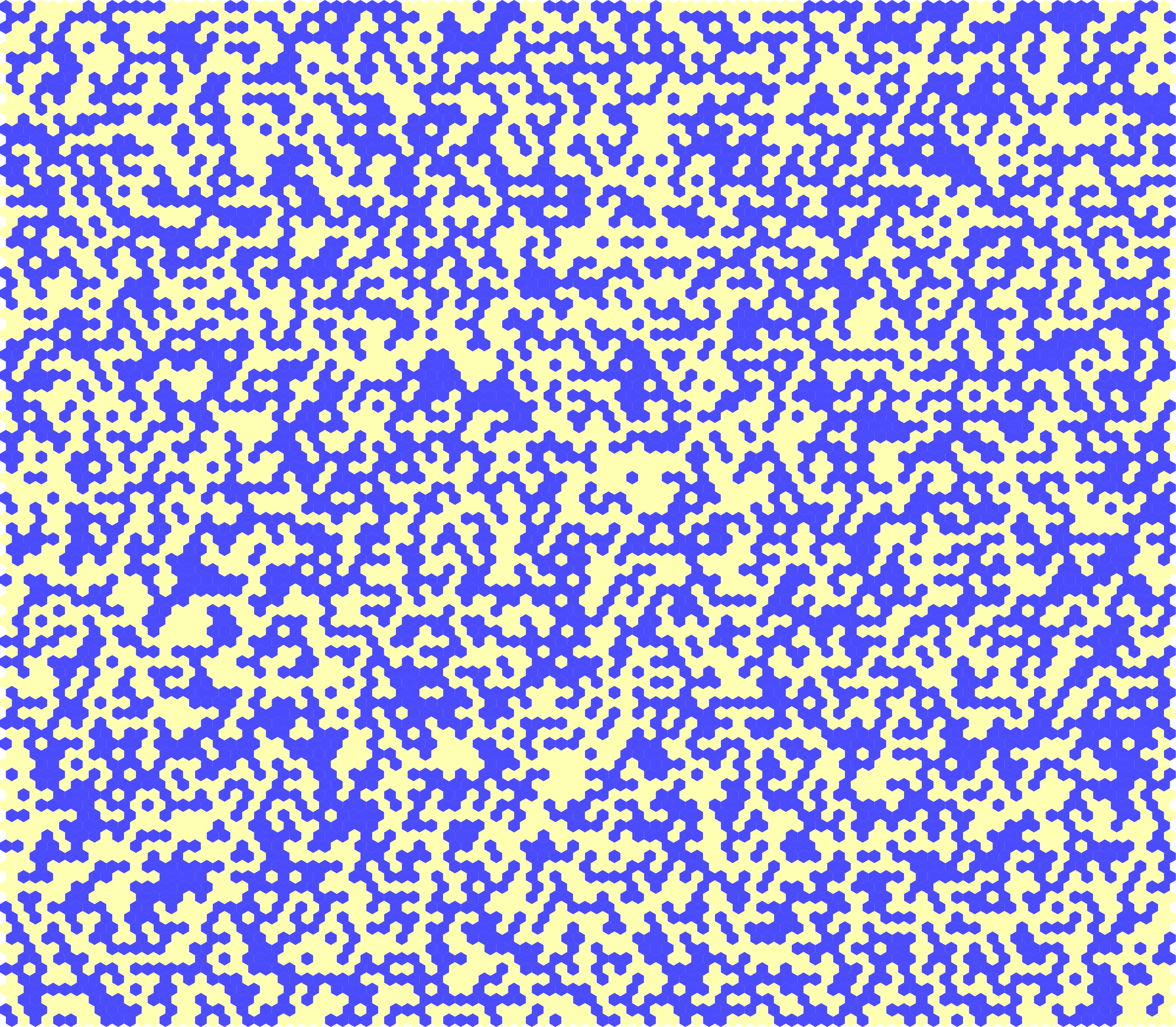}}
\caption{Site percolation on $T$ as face percolation on $H$}
\label{fig:tri.hex}
\end{figure}

Smirnov's theorem is for site percolation on $T$, shown with its dual hexagonal lattice $H$ in Fig.~\ref{fig:tri.hex}. Each vertex $v\in T$ corresponds to a hexagon $H_v$. A {\bf triangular vertex} means a vertex of $T$, and a {\bf triangular path} means the polygonal interpolation of a path in $T$, regarded as a polygonal curve in $\C$. Define similarly {\bf hexagonal vertex} and {\bf hexagonal path}.

Recall that a {\bf domain} is a nonempty proper open subset $D\subset\C$ which is simply connected. $D$ is called a {\bf Jordan domain} if its boundary $\pd D$ is a Jordan curve. In this case if $a,b\in \pd D$ then $\ol{ab}$ denotes the counterclockwise arc from $a$ to $b$ on $\pd D$.

\begin{dfn} \label{dfn:marked.domain}
A {\bf $k$-marked (Jordan) domain} is a Jordan domain $D$ together with $k$ distinct points $P_1,\ldots,P_k \in \pd D$ in counterclockwise order, denoted $D_k = (D;P_1,\ldots,P_k)$. Write $A_j\equiv A_j(D_k)$ for the arc $\ol{P_j P_{j+1}}$ (indices taken modulo $k$).

A {\bf conformal equivalence} of $k$-marked domains $D_k,D_k'$ is a conformal map $\vph:D\to D'$ with $\vph:P_j\mapsto P_j'$ for all $j$.
\end{dfn}

\begin{dfn} \label{dfn:marked.domain.disc}
A {\bf $k$-marked discrete domain with mesh $\de$} is a $k$-marked domain $D_k^\de=(D^\de;P_1^\de,\ldots,P_k^\de)$ such that $D^\de$ is the interior of a union of closed hexagonal faces of $\de H$, and each $P_j^\de$ is a vertex of $\de H$ incident to a unique hexagon inside $D$.
\end{dfn}

Let $\P^\de_p$ denote the law of percolation at probability $p$ on $D_4^\de$ (more precisely, face percolation on the hexagons inside $D^\de$). We say a {\bf blue} (yellow) {\bf crossing} of $D_4^\de$ occurs if there is a path of blue (yellow) hexagons joining $A_1(D_4^\de)$ to $A_3(D_4^\de)$. Let $A_j^+$ denote the set of hexagons in the external boundary of $D^\de$ with at least one edge contained in $A_j$; it is sometimes helpful to think of the $A_j^+$ as having predetermined colors, and these are shown with darker shading in the figures.

We now clarify the sense in which the $D_4^\de$ approximate $D_4$ as $\de\decto0$. To do so we need a formal definition of ``curves modulo reparametrization'':

\begin{dfn} \label{dfn:curve}
A distance function on the space $\cC_0$ of continuous functions $f:[0,1]\to\R^d$ is given by
$$\dcurve(f_1,f_2) \equiv \inf_\phi \nrm{ f_1\circ\phi-f_2 }_\infty$$
where the infimum is taken over increasing homeomorphisms $\phi$ of $[0,1]$. Say that $f_1,f_2$ are {\bf equivalent up to reparametrization}, denoted $f_1\sim f_2$, if $\dcurve(f_1,f_2)=0$. A {\bf curve} is an element of the space $\cC \equiv \cC_0/\sim$, and we refer to the metric $\dcurve$ on $\cC$ as the {\bf uniform metric}.
\end{dfn}

The space $\cC$ is separable and complete. For $\gam\in\cC$ we will frequently abuse notation and write $\gam$ also for a representative of $\gam$ in $\cC_0$.

\begin{dfn} \label{dfn:domain.unif.conv}
For bounded $k$-marked domains $D_k^\de,D_k$, say that $D_k^\de$ {\bf converges uniformly} to $D_k$ if $A_j(D_k^\de)\to A_j(D_k)$ uniformly for all $j$.
\end{dfn}

Recall that the critical probability for site percolation on $T$ is $p_c=1/2$; from now on we write $\P^\de\equiv \P^\de_{1/2}$. Here then is a statement of Smirnov's theorem confirming the Langlands et al.\ conjecture:

\begin{thm}[Smirnov's theorem on crossing probabilities \cite{\smirnovperc,\smirnovpercarxiv}]
\label{thm:smirnov}
Let $D_4^\de$ be four-marked discrete domains converging uniformly to the bounded four-marked Jordan domain $D_4$. Then
$$\Phi^\de(D_4^\de)
\equiv \P^\de(D_4^\de \text{ has a blue crossing})$$
converges to a limit $\Phi(D_4)\in(0,1)$ which is conformally invariant.
\end{thm}

\begin{rmk} \label{rmk:admiss}
It will be clear from the proof that the above assumption of uniform convergence is unnecessarily strong. In fact, in order to prove the scaling limit for the exploration path we will require a version of Thm.~\ref{thm:smirnov} which holds for a more general notion of discrete approximation, which will be conceptually straightforward but slightly tricky to describe. With a view towards keeping the exposition simple, we ignore the issue for now and address it in \S\ref{ssec:admiss}.
\end{rmk}

\subsection{FKG, BK, and RSW inequalities}
\label{ssec:fkg.bk.rsw}

In this section are collected some results of basic percolation theory which will be needed in the proof.

For the first two results, the FKG and BK inequalities, the graph structure is irrelevant so we return to a more general setting: let $\Om\equiv\set{0,1}^V$, and let $\P_p$ denote the law of site percolation at probability $p$ on $V$. For $\om,\om'\in\Om$, say $\om\le\om'$ if the inequality holds coordinate-wise. A random variable $X$ on $(\Om,\cF)$ is called {\bf increasing} if $X(\om)\le X(\om')$ for all $\om\le\om'$, and an event $A\in\cF$ is increasing if $\I_A$ is increasing. The following inequality tells us that increasing events are, as naturally expected, positively correlated:

\newtheorem*{thm:fkg}{FKG inequality}

\begin{thm:fkg}[Harris \cite{\harrisineq}, Fortuin, Kasteleyn, Ginibre \cite{\fkg}]
If $X,Y$ are increasing random variables on $(\Om,\cF)$, then $\E_p(XY)\ge(\E_p X)(\E_p Y)$. In particular, increasing events are positively correlated.
\end{thm:fkg}

Here is another useful result which gives bounds in the other direction: for $A,B\in\cF$ increasing events depending only on the states of finitely many sites, let $A \square B$ denote the event that there are {\bf disjoint witnesses} for $A$ and $B$ --- i.e., $\om \in A \square B$ if and only if there exist disjoint sets $I_A,I_B\subset V$ such that $\om'|_{I_A} = \om|_{I_A}$ implies $\om'\in A$, and $\om'|_{I_B} = \om|_{I_B}$ implies $\om'\in B$. The following inequality says that the existence of disjoint witnesses for two events is less likely than the simple intersection of the two events:

\newtheorem*{thm:bk}{BK inequality}

\begin{thm:bk}[van den Berg, Kesten \cite{\bk}]
If $A,B\in\cF$ are both increasing events depending only on the states of finitely many edges, then $\P_p(A \square B) \le \P_p(A) \P_p(B)$.
\end{thm:bk}

We remark that van den Berg and Kesten conjectured that their inequality could be generalized to arbitrary sets; this remained open for almost a decade until it was resolved by Reimer \cite{\reimer}.

The third result is specific to planar percolation. Russo \cite{\russo} and Seymour and Welsh \cite{\swperc} proved that in an $m\times n$ discrete rectangle, the probability of an open crossing between the length-$m$ sides has a non-zero lower bound depending only on the aspect ratio $m/n$. The following is a straightforward consequence of their result:

\begin{thm}
\label{thm:rsw}
There exist positive constants $c,\lm$ such that the $\P^\de$-probability of a blue crossing of an annulus with inner radius $r$ and outer radius $R$ is $\lesssim (r/R)^\lm$ provided $R/r,r/\de\ge c$.
\end{thm}

For proofs of the above results see \cite{\grimmett,\brperc}. A proof of Thm.~\ref{thm:rsw} can also be found in \cite{\wernerparkcity}.

\subsection{Preharmonic functions}
\label{sec:perc.harmonicity}

We now arrive at the central argument of Smirnov's proof. By way of historical background Smirnov mentions the classical connection of Kakutani \cite{\kakutani} between Brownian motion and harmonic functions described in \S\ref{sec:bm.harm}. The key to Smirnov's proof was the discovery of a ``preharmonic'' function $s_2^\de$ on $D_2^\de$, the ``separating probability function,'' which encodes the crossing probability. One then shows that as $\de\decto0$, $s_2^\de$ converges to a harmonic function $s_2$ which encodes the limiting crossing probability. The result follows because $s_2$ is a conformal invariant.

The function $s_2^\de$ is part of a triple $(s_1^\de,s_2^\de,s_3^\de)$ which we now define: let $D_4^\de$ be a four-marked domain, and let $D_3^\de$ be the three-marked domain obtained by forgetting $P_4^\de$. For $z\in D^\de\cap H$ and $1\le j\le 3$, let
$$E_j(z) \equiv E_j(z; D^\de_3)$$ denote the event that there is a blue $A_{j-1}$--$A_{j+1}$ \emph{simple} path separating $z$ from $A_j$ (indices taken modulo 3). A schematic diagram is shown in Fig.~\ref{fig:sep}. The {\bf separating probability functions} for $D_3^\de$ are
\beq \label{eq:sep.prob}
s^\de_j(z) \equiv \P^\de[ E_j(z) ],\quad 1\le j\le 3.
\eeq
In particular, notice that $s_2^\de(P_4^\de)$ is exactly the crossing probability $\Phi^\de(D_4^\de)$. The plan is then to prove that $s_2^\de$ converges to a function $s_2$ which is a conformal invariant of $D_3$. But $s_2(P_4)$ is the desired limiting crossing probability $\Phi(D_4)$, which must then be a conformal invariant.

\begin{figure}
\centering
\includegraphics[height=0.5\textwidth]{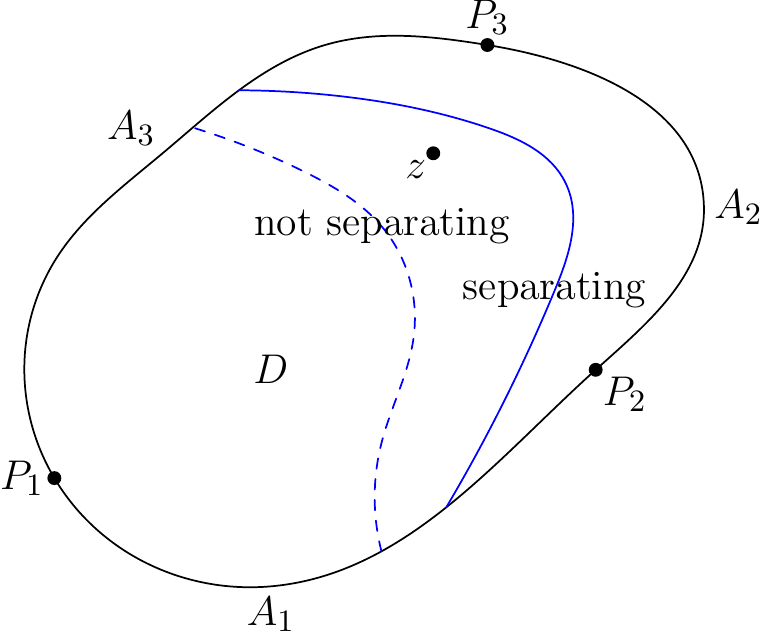}
\caption{Schematic of separating event $E_2(z)$}
\label{fig:sep}
\end{figure}

\subsubsection{Cauchy-Riemann equations}

The approach is to take advantage of the relationships among the $s_j^\de$ to show that they form a (discrete) ``harmonic conjugate triple.'' This concept is easier to describe in the continuous setting: we will say that real-valued harmonic functions $h_1,h_2,h_3$ form a {\bf harmonic conjugate triple} if $f=h_1+\ze h_2 + \ze^2 h_3$ is holomorphic, where $\ze\equiv\ze_3\equiv e^{2\pi i/3}$.

Here is a version of the Cauchy-Riemann equations for harmonic conjugate triples: consider the directional derivatives
$$\pd_k f(z) \equiv \lim_{h\to0} \f{f(z+h\ze^{k-1})-f(z)}{h},\quad
    1\le k\le 3.$$
The same argument used in deriving the usual Cauchy-Riemann equations (see e.g.\ \cite[p.~12]{\sscplx}) gives
\beq \label{eq:match.three.derivs}
\pd_k f = \frac{1}{\ze} \pd_{k+1} f = \frac{1}{\ze^2} \pd_{k+2} f,
\eeq
It will turn out that $h_1+h_2+h_3\equiv1$, so the $\pd_k h_j$ are uniquely determined by the real linear relations
$$\begin{pmatrix} \pd_k \real f \\ \pd_k \imag f \\ 0 \end{pmatrix}
= \begin{pmatrix}
    1 & \real \ze & \real \ze^2 \\
    1 & \imag \ze & \imag \ze^2 \\
    1 & 1 & 1
\end{pmatrix}
\begin{pmatrix} \pd_k h_1 \\ \pd_k h_2 \\ \pd _k h_3\end{pmatrix}.$$
Matching coefficients in \eqref{eq:match.three.derivs} gives the $2\pi/3$-rotational Cauchy-Riemann equations
\beq \label{eq:three.cauchy.riemann}
\pd_k h_j
= \pd_{k+1} h_{j+1}
= \pd_{k+2} h_{j+2},
\quad 1\le k,j\le 3,
\eeq
with indices taken modulo three.

We return now to the discrete setting. Recalling the definition \eqref{eq:sep.prob}, for $w\sim z$ in $\de H\cap D^\de$ define
$$d_z s_j^\de(w)
\equiv \P^\de[ E_j(z) \setminus E_j(w) ].$$
The following result is Smirnov's ``color switching identity,'' and is a discrete version of the $2\pi/3$-rotational Cauchy-Riemann equations \eqref{eq:three.cauchy.riemann}.

\begin{figure}
\subfloat[$E_3^\de(z_3) \setminus E_3^\de(w) = B_1 B_2 Y_3$]{\label{fig:switch.bbw}\includegraphics[height=0.37\textwidth]{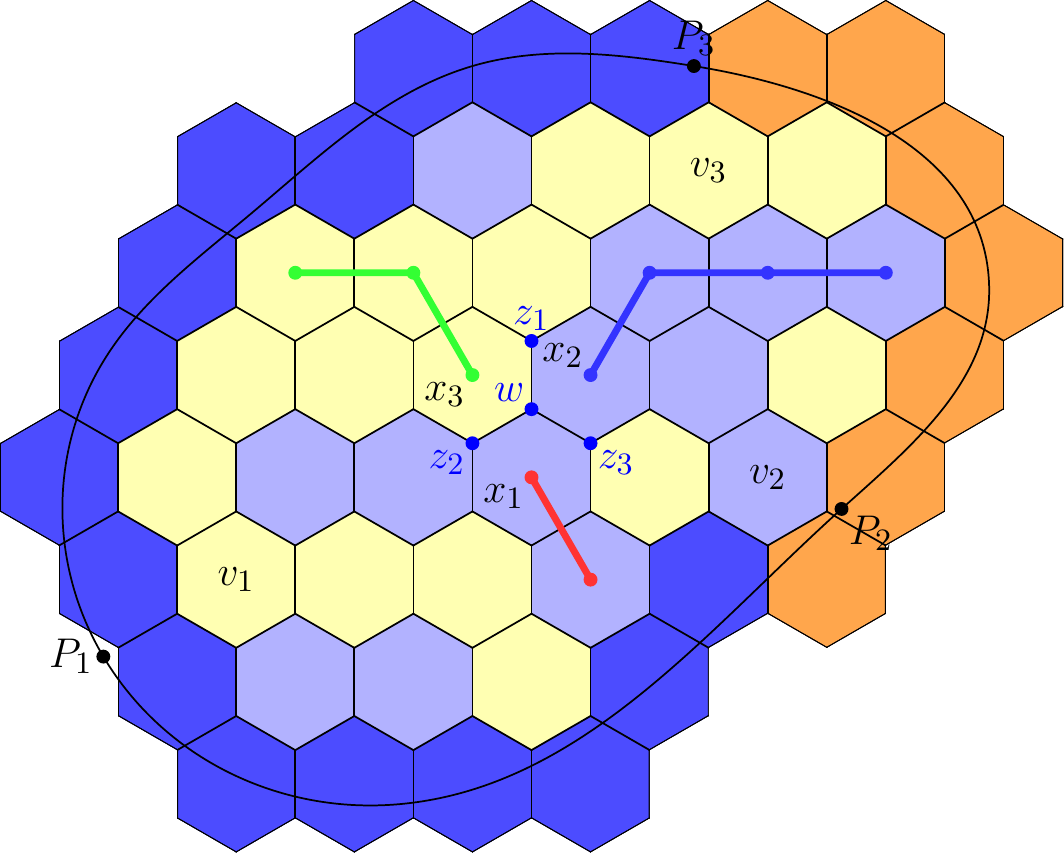}}
\qquad
\subfloat[Interface; ``innermost'' disjoint paths]{\label{fig:switch.inner}\includegraphics[height=0.37\textwidth]{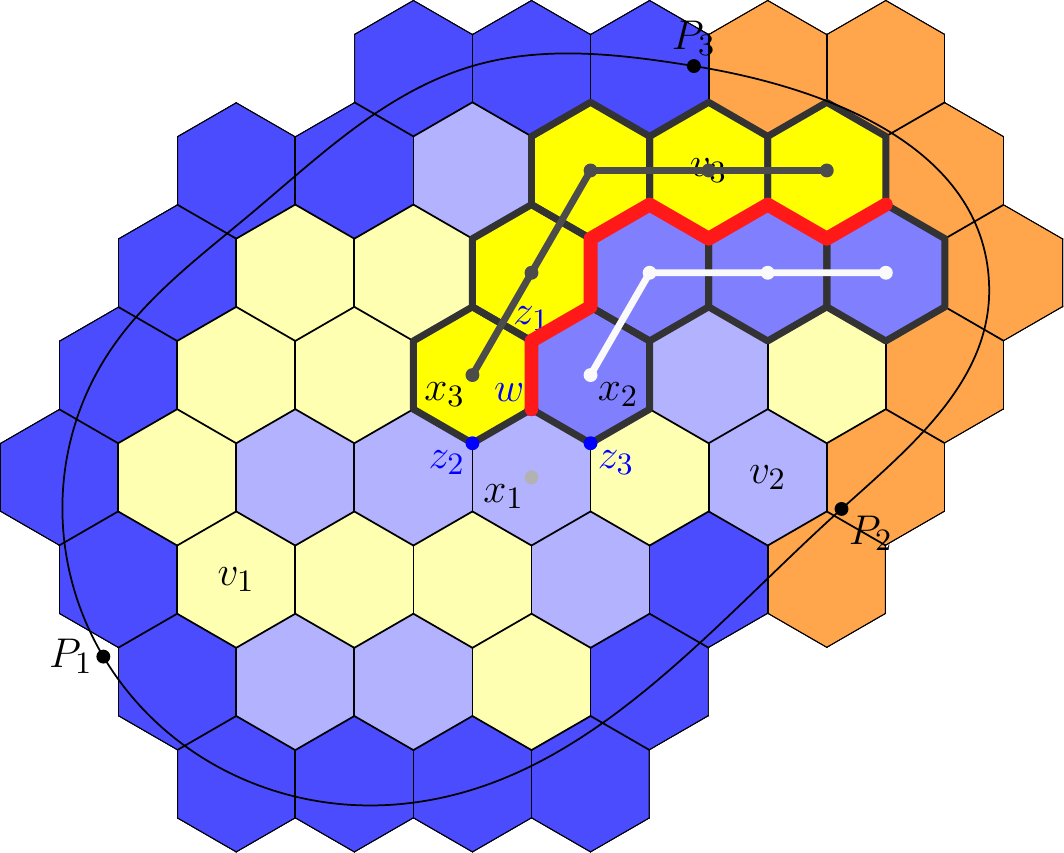}}
\caption{Color switching}
\end{figure}

\begin{ppn}[color switching] \label{ppn:color.switch}
In the setting of Thm.~\ref{thm:smirnov}, let $w\in \de H \cap D^\de$, and fix a counterclockwise ordering $z_1,z_2,z_3$ of the neighbors of $w$ in $\de H$, assumed to lie in $D^\de$. Then
\beq \label{eq:color.switch}
d_{z_k} s_j^\de(w)
= d_{z_{k+1}} s_{j+1}^\de(w)
= d_{z_{k+2}} s_{j+2}^\de(w).
\eeq
Further $d_z s_j^\de(w) = O(\de^\lm)$ where $\lm$ is as in Thm.~\ref{thm:rsw}.

\begin{proof}
For now we suppress the dependence on $\de$ from the notation. By cyclic permutation of indices it suffices to prove the result with $k=j$. Let $x_j\in T$ be the center of the hexagon incident to $w$ opposite $z_j$, and let $C_j$ denote the event that there is a path $\eta_j$ of color $C\in\set{B=\text{blue},Y=\text{yellow}}$ joining $x_j$ to $A_j$. Then
$$E_j(z_j) \setminus E_j(w) = B_{j-1} Y_j B_{j+1},$$
where the right-hand side denotes the event that $B_{j-1} \cap Y_j \cap B_{j+1}$ occurs via \emph{disjoint} paths $\eta_{j-1}$, $\eta_j$, $\eta_{j+1}$. Fig.~\ref{fig:switch.bbw} illustrates this equality of events; a detailed proof of this statement may be found in \cite[Ch.~7]{\brperc}. Thus \eqref{eq:color.switch} may be restated as
$$\P(Y_1 B_2 B_3)=\P(B_1 Y_2 B_3)=\P(B_1 B_2 Y_3).$$
By symmetry it suffices to prove the second identity, which by flipping blue and yellow is the same as $\P(Y_1B_2Y_3)=\P(B_1B_2Y_3)$.

Conditioned on $B_2 Y_3$, consider exploring the interface beginning at the edge separating $H_{x_2}$ (blue) from $H_{x_3}$ (yellow), oriented so that blue is to the right. This exploration path $P$ must end on some vertex of $H \cap (A_2 \cup A_3)$, and the hexagons adjacent to $P$ determine the ``innermost'' disjoint paths $\eta_2$ and $\eta_3$ corresponding to $B_2$ and $Y_3$ respectively (erase loops so that the $\eta_j$ are simple paths). Note that these paths are determined without looking at any sites not incident to $P$.

Conditioning on the explored sites, the event $C_1B_2Y_3$ occurs if and only if there is a path of color $C$ from $x_1$ to $A_1$ which is \emph{disjoint} from the explored sites. But the probability of this event clearly does not depend on whether $C$ is blue or yellow, so
$\P(Y_1B_2Y_3|B_2Y_3)=\P(B_1B_2Y_3|B_2Y_3)$ and \eqref{eq:color.switch} is proved.

Finally, since the maximum distance from $w$ to an arc $A_k$ is of constant order, it follows from Thm.~\ref{thm:rsw} that $d_z s_j^\de(w) = O(\de^\lm)$.
\end{proof}
\end{ppn}

\begin{rmk}
In fact $d_z s_j^\de(w) = O(\de^{2/3})$; see \cite{\smirnovpercarxiv}. However, we will see that the $s_j^\de$ converge to harmonic functions, which suggests that the discrete derivatives
$$s_j^\de(z)-s_j^\de(w)=d_z s_j^\de(w) - d_w s_j^\de(z)$$
are in fact $O(\de)$ --- that is, there is some cancellation between the disjoint events $E_j(z)\setminus E_j(w)$ and $E_j(w)\setminus E_j(z)$. It is unclear why this happens, and we turn now instead to proving relations among contour integrals, where the estimate of Propn.~\ref{ppn:color.switch} is sufficient.
\end{rmk}

\subsubsection{Integral relations}
\label{ssec:rel.integrals}

In taking the limit $\de\decto0$ it is easier to use global manifestations of holomorphicity such as Morera's theorem (see e.g.\ \cite[Thm.~5.1]{\sscplx}). In this section we use Propn.~\ref{ppn:color.switch} to prove relations on discrete contour integrals.

Suppose $(h_1,h_2,h_3)$ form a harmonic conjugate triple: since $f$ and $1$ are both holomorphic, by Morera's theorem
$$\begin{pmatrix} 0 \\ 0\end{pmatrix}
=\begin{pmatrix} 1 & \ze & \ze^2 \\
    1 & 1 & 1
\end{pmatrix} \begin{pmatrix} I_1 \\ I_2 \\ I_3 \end{pmatrix},\quad
I_j \equiv \oint_C h_j \d z,
$$
therefore $\oint_C h_{j+1} \ dz = \ze \oint_C h_j \ dz$. It is instructive to note that this relation can alternately be derived from the $2\pi/3$-rotational Cauchy-Riemann equations \eqref{eq:three.cauchy.riemann}, using Stokes's theorem. To see this, write
$$dx_1 \equiv dx,\quad
dx_2 \equiv -\frac{1}{2} \d x +  \frac{\sqrt{3}}{2} \d y,\quad
dx_3 \equiv -\frac{1}{2} \d x -  \frac{\sqrt{3}}{2} \d y.$$
Then
\begin{align*}
\sum_{k=1}^3 \pd_k h_j \d x_k
&= \f{3}{2} \lb \pd_x h_j \d x + \pd_y h_j \d y \rb
= \f{3}{2} dh,\\
\sum_{k=1}^3 \ze^{k-1} \d x_k
&= \f{3}{2} \lb dx + i \d y \rb = \f{3}{2} dz.
\end{align*}
Thus, for a contour $C=\pd U$, Stokes's theorem implies
$$\oint_C h_j \ dz
= \frac{4}{9} \int_U
    \lp
        \sum_{k=1}^3 \pd_k h_j \d x_k
    \rp
    \wedge
    \lp \sum_{k=1}^3 \ze^{k-1} dx_k \rp.$$
Applying the Cauchy-Riemann equations \eqref{eq:three.cauchy.riemann} then gives
\begin{align*}
&\oint_C h_j \ dz
= \frac{4}{9}  \int_U
    \lp \sum_{k=1}^3 \pd_{k+1} h_{j+1} \d x_k \rp
\wedge \lp \sum_{k=1}^3 \ze^{k-1} dx_k \rp \\
&= \f{1}{\ze} \cdot \f{4}{9}
    \int_U \lp \sum_{k=1}^3  \pd_k h_{j+1} \d x_{k-1} \rp \wedge
    \lp \sum_{k=1}^3 \ze^{k-1} \d x_{k-1} \rp
= \frac{1}{\ze} \oint_C h_{i+1} \ dz,
\end{align*}
the same relation derived above.

We will now prove a discrete analogue of this relation for certain contours; the proof will be a discrete version of the Stokes' theorem computation above. We will say that a {\bf discrete triangular contour} is a contour which is a triangular path. For a triangular contour $C$ oriented counterclockwise in $D^\de$, we define the \textbf{discrete contour integral}
$$\doint{C} s_j^\de(z) \ dz = \sum_{e \in C} s_j^\de(z(e)) (e^+-e^-),$$ where $z(e)$ is the vertex of $H$ immediately to the left of the edge $e=(e^-,e^+)$.

\begin{ppn}
\label{ppn:rotate.int}
In the setting of Thm.~\ref{thm:smirnov}, if $C$ is a discrete triangular contour in $D^\de$ of length $L$, then
$$\doint{C} s_{j+1}^\de(z) \d z - \ze \doint{C} s_j^\de(z) \d z
= O( \de^\lm  L)$$
where $\lm$ is as in Thm.~\ref{thm:rsw}.

\begin{proof}
Let $\inter C$ denote the points of $\de H$ inside $C$, and use $1+\ze+\ze^2=0$ to write
$$\doint{C} s_j^\de(z) \d z
= i\sqrt{3} \sum_{w \in C^\circ} \sum_{z \sim w, z \notin C^\circ} s_j^\de(w) (z-w)
= i \sqrt{3} \sum_{w\in C^\circ} \sum_{z\sim w, z\in C^\circ}
    s_j^\de(w) (w-z).$$
This can be rewritten as
\begin{align*}
&\sum_{\set{w,z}} (z-w) [s_j^\de(z)-s_j^\de(w)]
= \sum_{\set{w,z}} (z-w) [ d_z s_j^\de(w) - d_w s_j^\de(z) ] \\
&= \sum_{w\in C^\circ} \sum_{z\sim w} d_z s_j^\de(w) (z-w)
    -\sum_{w\in C^\circ} \sum_{z\sim w, z\notin C^\circ} d_z s_j^\de(w) (z-w)\equiv S_j - E_j.
\end{align*}
Propn.~\ref{ppn:color.switch} implies $S_{j+1}=\ze S_j$ while $E_j = O(\de^\lm L)$, which proves the result.
\end{proof}
\end{ppn}

\subsection{Completing the proof}
\label{sec:perc.completing}

We can easily extend $s_j^\de$ to all of $\de H$ by setting $s_j^\de(z)$, $z\notin D^\de$, to equal $s_j^\de(z')$ where $z'$ is one of the hexagonal vertices inside $D^\de$ closest to $z$. Let $g_j^\de$ denote the extension to the closure of $D$ by linear interpolation. We will prove Thm.~\ref{thm:smirnov} by showing that the $g_j^\de$ converge uniformly to the triple characterized by the following lemma, due to Beffara:

\begin{lem}[Beffara \cite{\beffara}] \label{lem:unique.triple}
Let $D_3$ be a three-marked domain. There is unique triple of continuous real-valued functions $(h_1,h_2,h_3)$ defined on the closure of $D$ satisfying the following two conditions:
\begin{enumerate}[(i)]
\item For any triangular contour $C$ in $D$, $\oint_C h_{j+1}(z) \ dz = \ze \oint_C h_j(z) \ dz$.
\item For any $z \in A_j$, $h_j(z)=0$ and $h_{j+1}(z)+h_{j+2}(z)=1$.
\end{enumerate}
Moreover these functions are harmonic, hence a conformal invariant of $D_3$.

\begin{proof}
Property (i) and Morera's theorem give that $f=h_1+\ze h_2 + \ze^2 h_3$ and $h=h_1+h_2+h_3$ are holomorphic, hence the $h_j$ are harmonic. Property (ii) and the maximum principle then gives $h\equiv1$.

On the arc $A_j$, property (ii) implies that $f$ is a convex combination of $\ze^j$ and $\ze_{j+1}$. Thus, as $z$ travels along $\pd D$, $f(z)$ winds around $\pd\De$ where $\De$ is the equilateral triangle with marked vertices $1,\ze,\ze^2$. It follows from the argument principle (see e.g.\ \cite[Thm.~4.1]{\sscplx}) that $f$ is unique conformal equivalence of $D_3$ with $\De$. The result follows since the $h_j$ are uniquely determined by $(f,h)$.
\end{proof}
\end{lem}

\begin{proof}[Proof of Thm.~\ref{thm:smirnov}]
It is clear that the $g_j^\de$ are uniformly bounded, and uniform equicontinuity follows from the bound on the discrete derivatives (Propn.~\ref{ppn:color.switch}), so the $g_j^\de$ have subsequential uniform limits by the Arzel\`a-Ascoli theorem. We must show that any subsequential limit satisfies the characterization of Lem.~\ref{lem:unique.triple}. Indeed, property (i) follows from Propn.~\ref{ppn:rotate.int} by approximating $C$ with discrete triangular contours, using uniform convergence to the limit and the uniform equicontinuity of the $g_j^\de$. Property (ii) is intuitively clear, and is easy to prove rigorously using Thm.~\ref{thm:rsw}.

It follows that
$\Phi^\de(D_4^\de) = g_2^\de(P_4^\de)$ converges to $\Phi(D_4)=h_2(P_4)$, which is a conformal invariant of $D_4$ by Lem.~\ref{lem:unique.triple}. This concludes the proof of Smirnov's theorem.
\end{proof}

Carleson noted that Cardy's formula takes a particularly simple form when the domain is an equilateral triangle. A consequence of the above proof is the verification of this formula for critical percolation on $T$:

\newtheorem*{thm:cardy}{Cardy's formula in Carleson's form}

\begin{thm:cardy}[Smirnov \cite{\smirnovperc,\smirnovpercarxiv}, Beffara \cite{\beffara}]
Let $D_3$ denote the three-marked domain obtained from $D_4$ by forgetting $P_4$, and $f$ the conformal equivalence $D_3\to\De$. Then
$\Phi(D_4)=x(D_4)$ where $x\equiv x(D_4)$ is defined by
$$f(P_4) = x \ze^2 + (1-x).$$

\begin{proof}
For $x\in A_3(\De)$ write $\De_x\equiv (\De;1,\ze,\ze^2,x)$: Thm.~\ref{thm:smirnov} implies
$$\Phi(D_4) = \Phi(\De_{f(P_4)}) = h_2(f(P_4))$$
where $(h_1,h_2,h_3)$ is the triple on $\De$ given by Lem.~\ref{lem:unique.triple}. But $h_1+\ze h_2 + \ze^2 h_3$ must simply be the identity on $\De$, so $h_2$ is linear on $A_3(\De)$ and the result follows.
\end{proof}
\end{thm:cardy}

In forthcoming work Mendelson, Nachmias, Sheffield and Watson prove an $O(\de^\ep)$ bound on the rate of convergence to Cardy's formula \cite{\mnwcardy}.

For two more percolation formulas of a similar nature see \cite{\schrammperc,\dubedatwatts,\schrammwatts}.

\subsection{The percolation exploration path}
\label{sec:perc.end}

Let $D^\de_2\equiv (D^\de;a^\de,b^\de)$ be a two-marked domain, and suppose $A_1^+,A_2^+$ are colored yellow, blue respectively. The {\bf percolation exploration path} $\gam^\de$ is the portion of the blue-yellow interface traveling from $a^\de$ to $b^\de$. The purpose of the remainder of this article is to prove the following:

\begin{thm}[Schramm \cite{\schramm}, Smirnov \cite{\smirnovpercarxiv, \smirnovperc}, Camia and Newman \cite{\cmnw}]
\label{thm:explore}
Let $\gam^\de$ denote the percolation exploration path in the discrete  domain $D_2^\de$. If $D_2^\de$ converges uniformly to the two-marked bounded Jordan domain $D_2$, then $\gam^\de$ converges weakly to chordal $\sle_6$ in $D_2$.
\end{thm}

\begin{rmk} \label{rmk:wk.conv}
For random variables on a separable metric space, there are several equivalent ways to define convergence in law (weak convergence, convergence in distribution): in our setting, a natural formulation is via the {\bf Skorohod coupling theorem} (more precisely, the generalization due to Dudley \cite{\skorohoddudley}) which states that random variables converge in law if and only if they can be defined on a joint probability space in which they converge almost surely.
\end{rmk}

In \S\ref{ch:sle} we define $\sle$ and give some characterizations of $\sle_6$. \S\ref{ch:scaling} shows that the $\gam^\de$ have \emph{subsequential} weak limits. Finally, \S\ref{ch:explore} shows that all subsequential limits coincide with chordal $\sle_6$, concluding the proof.

\section{Schramm-Loewner evolutions}
\label{ch:sle}

In this chapter we derive the chordal Loewner differential equation
$$\dot g_t(z) = \f{2}{g_t(z)-u_t},\quad g_0(z)=z,\quad u_t=g_t(\gam(t))$$
for a self-avoiding (see Defn.~\ref{dfn:self.avoiding} below) curve traveling from $0$ to $\infty$ in $\clos\H$. We then show that the $\sle$ axioms imply that $u_t$ is a Brownian motion. We conclude with a characterization of $\sle_6$ which will be used in the proof of Thm.~\ref{thm:explore}. A fully rigorous treatment of $\sle$ is beyond the scope of this article, but we attempt to highlight the main points.

\subsection{Normalizing the conformal maps}

By conformal invariance it suffices to define chordal $\sle$ for a single two-marked domain, and it turns out a convenient choice is $\H_2\equiv(\H;0,\infty)$. In this section \emph{only}, we take $\C$ with the spherical metric, so that $\H$ is a bounded domain with compact closure $\clos\H=\H\cup\clos\R$. We will begin by describing the Loewner slit mapping theorem for a simple curve $\gam$ traveling chordally in $\H_2$, and extend to more general curves afterwards.

Recall that conformal automorphisms of a simply connected domain are determined up to three real degrees of freedom, for example, the group of conformal automorphisms of $\H$ is the group of M\"obius transforms, $\Aut\H\cong\SL_2\R$ (see e.g.\ \cite[Thm.~2.4]{\sscplx}). Thus conformal equivalences between two-marked domains are not uniquely determined, so we must choose an appropriate normalization for the maps $g_t$.

From now on we write $f\lesssim g$ if $f/g$ is bounded by a (universal) constant, and $f \asymp g$ if both $f\lesssim g$ and $g \lesssim f$.

\subsubsection{Boundary behavior}

We begin with a caution that while conformal maps are as nice as possible in the interiors of domains, their good behavior does not necessarily continue up to the boundary. The results of this section may be found in \cite{\pommerenke} which has much more information on this subject.

\newtheorem*{thm:cara}{Carath\'eodory's theorem}
\begin{thm:cara}
Let $f : \D \to D$ be a conformal map. The function $f$ has a continuous extension to $\clos\D$, which restricts to a bijection of $C=\pd\D$ with $\pd D$, if and only if $\pd D$ is a Jordan curve.
\end{thm:cara}

Unfortunately, Carath\'eodory's theorem does not apply in our main case of interest, where $D$ is a ``slit'' domain of the form $\H\setminus\gam[0,t]$. We do, however, have the following result:

\begin{dfn} \label{dfn:loc.conn} The closed set $A \subseteq \C$ is called {\bf locally connected} if for all $\ep>0$ there exists $\de>0$ such that, for any two points $a,b \in A$ with $|a-b|<\de$, we can find a continuum\footnote{A nonempty compact connected subset of $\C$.} $B\subseteq A$ with $a,b \in B$ and $\diam B < \ep$.
\end{dfn}

\newtheorem*{thm:cty}{Continuity theorem}

\begin{thm:cty}
Let $f : \D \to D$ be a conformal map. The following are equivalent:
\begin{enumerate}[(i)]
\item The function $f$ has a continuous extension to $\clos\D$;
\item $\pd D$ is a curve;
\item $\pd D$ is locally connected;
\item $\C \setminus D$ is locally connected.
\end{enumerate}
\end{thm:cty}

This tells us that the maps are well-behaved at least in one direction, and in the other direction more care is needed. In particular, the driving function ``$u_t=g_t(\gam(t))$'' is not {\it a priori} well-defined.

\subsubsection{Half-plane capacity}

We define a {\bf compact $\H$-hull} to be a bounded subset $A\subset\H$ with $A$ closed in $\H$ and $\H\setminus A$ simply connected ($A$ itself is not required to be connected). The {\bf radius} of $A$, denoted $\rad A$, is the radius of the smallest closed disc centered at the origin which contains $A$. Let $\cQ$ denote the collection of compact $\H$-hulls.

We will normalize a conformal map $g_A:\H\setminus A\to\H$ by the requirement that it behave ``like the identity'' near $\infty$. The following proposition makes this precise:

\begin{ppn} \label{ppn:def.hcap}
For $A\in\cQ$, there is a unique conformal transformation $g_A : \H \setminus A \to \H$ such that $$\lim_{z \to \infty} \left[ g_A(z)-z\right]=0.$$ We say that $g_A$ has the {\bf hydrodynamic normalization}.

\begin{proof}
The inversion map $j : z \mapsto -1/z$ is an element of $\Aut\H$. $D=j(\H\setminus A) \subseteq \H$ is simply connected and contains the intersection of $\H$ with some neighborhood of 0. By the Riemann mapping theorem and Carath\'eodory's theorem, there exists a conformal map $\phi : D \to \H$ with $\phi(0)=0$, extending continuously to the boundary of $D$ near 0. By Schwarz reflection (see \cite[\S4.6.5]{\ahlfors}) $\phi$ has a power series expansion
$$\phi(z) = a_1z + a_2z^2 + a_3 z^3 + \cdots$$
around 0 with $a_1>0$ and $a_k \in \R$ for all $k$, since $\phi$ maps the real line near 0 to the real line. The map $\phi$ is determined up to composition with any element of $\Aut\H\cong\SL_2\R$ fixing $0$, and a simple calculation shows that $\phi$ is fully determined by the choice of $a_1>0$, $a_2\in\R$. Define $g_A(z)\equiv j\circ\phi\circ j(z)$; its expansion around $\infty$ is
$$g_A(z) \equiv \f{-1}{\phi(-1/z)}
= \frac{z}{a_1} + \frac{a_2}{a_1^2} + \left( \frac{a_2^2}{a_1^3} - \frac{a_3}{a_1^2} \right) \frac{1}{z} + \cdots
\equiv b_{-1} z + b_0 + \frac{b_1}{z} + \cdots.$$
The result follows by choosing $a_1=1$, $a_2=0$.
\end{proof}
\end{ppn}

\begin{dfn} \label{dfn:hcap}
If $A$ is a compact $\H$-hull, the {\bf half-plane capacity} (from $\infty$) of $A$ is
$$\hcap A = \lim_{z \to \infty} z\left[g_A(z)-z\right].$$
\end{dfn}

In the notation of the proof of Propn.~\ref{ppn:def.hcap},
$\hcap A=b_1$ and
$$g_A(z) = z + \frac{\hcap A}{z} + O\left(\frac{1}{|z|^2} \right) \ \text{ as } z \to \infty.$$
If $r>0$ then $g_{rA}(z) = r g_A(z/r)$ so $\hcap (rA) = r^2 \hcap A$, and if $x \in \R$ then $g_{A+x} = g_A(z-x)+x$ so $\hcap(A+x)=\hcap A$. If $A,B\in\cQ$ with $A\subseteq B$ then $g_B = g_{g_A(B\setminus A)} \circ g_A$, and expanding gives
\beq \label{eq:hcap.add}
\hcap B = \hcap A + \hcap (g_A(B \setminus A)).
\eeq
This suggests that the half-plane capacity is a measure of the size of $A$, but we have not yet shown that $\hcap$ is positive. The following result proves this and gives a more precise characterization of the capacity in terms of Brownian motion.

\begin{ppn} \label{ppn:hcap.bm}
Let $A\in\cQ$, $\bm_t$ a Brownian motion started at $z\in\H\setminus A$, and $\tau\equiv\tau_{\H\setminus A}$ the leaving time of $\H\setminus A$. Then
$\imag[z-g_A(z)]=\E_z (\imag \bm_\tau)$.

\begin{proof}
By the hydrodynamic normalization, $\phi(z) = \imag [z-g_A(z)]$ is a bounded harmonic function on $\H \setminus A$, and can be extended continuously to the boundary by setting $\phi(z)=\imag z$ for $z\in\pd(\H\setminus A)$. Therefore $\phi(\bm_{t\wedge\tau})$ is a martingale, and the optional sampling theorem gives
$$\phi(z) = \E_z \phi(\bm_\tau)
= \E_z [\imag \bm_\tau],$$
proving the result.
\end{proof}
\end{ppn}

An immediate consequence is that
\beq \label{eq:hcap.lim.y}
\hcap A
= \lim_{y\to\infty} y \E_{iy}[\imag W_\tau],
\eeq
which proves that $\hcap$ is positive. For example, $g_{\clos\D\cap\H}(z) = z+z^{-1}$ so $\hcap(\clos\D\cap\H)=1$, implying the general bound $\hcap A\le(\rad A)^2$. For more geometric interpretations of $\hcap$ see \cite{\llnhcap}.

We now state two results which will be used in the LDE derivation; the proofs are deferred to \S\ref{sec:lde.sle.cap}. The first result is a uniform estimate on higher-order remainder terms in the Laurent expansion of $g_A(z)$:

\begin{ppn} \label{ppn:hcap.err}
There exists a constant $c<\infty$ such that
\beq\label{eq:hcap.err}
\left|g_A(z)-z-\frac{\hcap A}{z}\right| \le c \frac{(\rad A) (\hcap A)}{|z|^2}
\eeq
for all $|z| \ge 2 \rad A$.
\end{ppn}

The next result is a bound on the distortion of size under the maps $g_A$.

\begin{ppn} \label{ppn:lde.distort}
For $A\in\cQ$ and $B\subset\H$ such that $A\cup B\in\cQ$,
$$\diam g_A(B)
\lesssim [(\diam B) (\sup \set{\imag z:z\in B}) ]^{1/2}.$$
\end{ppn}

Assuming these results, in the next section we derive the LDE and use it to define the Schramm-Loewner evolutions.

\subsection{The Loewner differential equation and $\sle$}
\label{sec:lde.sle}

Returning to the original problem, let $\gam$ be a simple curve traveling from $0$ to $\infty$ in $\H$. Let
$$g_t\equiv g_{\gam(0,t]},\quad b(t)\equiv \hcap\gam(0,t].$$
Let $g_{s,t} \equiv g_{g_s(\gam(s,t])}$ so that $g_t = g_{s,t} \circ g_s$. The following is a consequence of Propn.~\ref{ppn:lde.distort}:

\begin{cor} \label{cor:lde.driving}
The map $g_t$ can be extended continuously to $u_t\equiv g_t(\gam(t))$ which is real-valued and continuous in $t$ with $u_0=0$.

\begin{proof}
Note that
\beq\label{eq:lde.supnorm}
\|g_s-g_t\|_\infty
= \sup_z \abs{ g_{s,t}(z)-z }
\lesssim \diam g_s(\gam(s,t]),
\eeq
and Propn.~\ref{ppn:lde.distort} implies
\beq \label{eq:lde.diam.est}
\diam g_s(\gam(s,t])
\lesssim [ (\diam \gam[s,t]) (\diam \gam[0,t]) ]^{1/2}
\eeq
Let $w_t\equiv \lim_{s\incto t} g_s(\gam(t))$; \eqref{eq:lde.supnorm} and \eqref{eq:lde.diam.est} imply that $w$ is well-defined and continuous in $t$. Given $\ep>0$, since $\gam$ is \emph{simple}, there exists $0<\de<\ep$ such that if $0\le s<t\le t_0$ and $\abs{\gam(s)-\gam(t)}<\de$, then $t-s<\ep$ and $\diam\gam[s,t]<\ep$. Let $B=\ball[\gam(t)]{\de}\cap (\H\setminus\gam[0,t])$: for any $z\in B$, there exists a simple curve $\tilde\gam$ in $B$ of diameter $<\de$ joining $z$ to some $\gam(t')\in\ball[\gam(t)]{\de}$ with $t'<t$, hence $t-t'<\ep$ and $\diam\gam[t',t]<\ep$. Then
\begin{align*}
|g_t(z)-w_t|
&\le |g_t(z)-g_{t'}(z)| + |g_{t'}(z)-g_{t'}(\gam(t))| + |g_{t'}(\gam(t))-w_t|\\
&\le \nrm{g_t-g_{t'}}_\infty + \diam g_{t'} (\tilde\gam)
    + \limsup_{s\incto t} \nrm{g_{t'}-g_s}_\infty,
\end{align*}
which by \eqref{eq:lde.supnorm} and \eqref{eq:lde.diam.est} can be made arbitrarily small by taking $\ep\decto0$. Therefore $u_t=w_t$ is continuous in $t$.
\end{proof}
\end{cor}

By \eqref{eq:hcap.add} and \eqref{eq:lde.diam.est}, the set $A=g_s(\gam(s,t])-u_s$ has half-plane capacity $b(t)-b(s)$ and radius $r_{s,t}$ which both approach zero as $t-s\decto0$. By reparametrizing, suppose $b$ is continuously differentiable. Propn.~\ref{ppn:hcap.err} gives
\begin{align*}
&g_t(z)-u_s
=g_{g_s(\gam(s,t])-u_s}(g_s(z)-u_s) \\
&= [g_s(z)-u_s] + \f{b(t)-b(s)}{g_s(z)-u_s}
    + O\lp \f{r_{s,t} [ b(t)-b(s)] }{|g_s(z)-u_s|^2}\rp,
\end{align*}
where the final term tends to zero faster than $t-s$ as $t-s\decto0$. Rearranging and taking $t-s\decto0$ gives
$$\dot g_t(z) = \frac{\dot b(t)}{g_t(z)-u_t}, \quad g_0(z)=z.$$
The standard parametrization\footnote{For a discussion of parametrization of $\sle$ see \cite{\lawlersheffield}.} then sets $b(t)=2t$, so that
\beq \label{eq:chordal.lde}
\dot g_t(z) = \frac{2}{g_t(z)-u_t}, \quad g_0(z)=z.
\eeq
This is the {\bf (chordal) Loewner differential equation} (LDE) with {\bf driving function} $u_t$.

In fact, the only place in the derivation where we used that $\gam$ was simple was in the proof of Cor.~\ref{cor:lde.driving}. In fact, using Propn.~\ref{ppn:lde.distort} one can see that the LDE applies to any {\bf continuously increasing hull process}: a strictly increasing process $K_t\in\cQ$ with $b(t)=\hcap K_t$ continuously differentiable, such that $g_t(K_{t+\de})$ decreases to a single point $u_t\in\R$ as $\de\decto0$ and $u_t$ is continuous in $t$. Fig.~\ref{fig:lde.fld} shows the vector field $2/[g_\D(z)-2]$ (corresponding to $K_t=\D$ and $u_t=2$).

\begin{figure}
\centering
\caption{Vector field $2/[g_\D(z)-2]$ around $\D$}
\label{fig:lde.fld}
\end{figure}

Conversely, given a continuous function $u_t:[0,\infty)\to\R$ with $u_0=0$, the LDE with driving function $u_t$ is simply an ODE, so we can solve this ODE to recover the increasing hull process: $g_t(z)$ is well-defined up to the first time $T_z$ that $\lim_{t\incto T_z} [g_t(z)-u_t]=0$. Consequently $K_t = \set{z\in\H:T_z\le t}$, the set of points ``swallowed'' by time $t$ --- note that multiple points can be swallowed at once.

\begin{dfn} \label{dfn:self.avoiding}
Let $\gam$ be any curve (not necessarily simple) traveling from $0$ to $\infty$ in $\clos\H$. The {\bf filling} of $\gam[0,t]$ is the set $K_t\in\cQ$ of all $z\in\H$ not belonging to the unbounded component of $\H\setminus\gam[0,t]$. We call $K_t$ the hull process {\bf generated} by $\gam$, and say that $\gam$ is {\bf self-avoiding} if $K_t$ is continuously increasing.
\end{dfn}

By abuse of notation we write $\hcap\gam[0,t]\equiv\hcap K_t$: if $\gam$ is self-avoiding then it has a parametrization with $\hcap\gam[0,t]=2t$, in which case we say $\gam$ is ``parametrized by its half-plane capacity.'' If a curve travels for a non-trivial time interval into its past hull or along the domain boundary, its half-plane capacity remains constant during this time so it is not parametrized by $\hcap$. There are also curves which may be parametrized by $\hcap$ but which do not generate a continuously increasing hull process, specifically curves with transversal self-crossings.

If the (hypothetical) random curve $\gam$ is to satisfy the $\sle$ axioms, its driving function must be a continuous process of stationary independent increments. By L\'evy's characterization of Brownian motion this implies $u_t= \mu t + \sqrt{\ka} \bm_t$, for $\bm_t$ a standard Brownian motion. If further the model has left-right symmetry then $\mu=0$.

\begin{dfn}
The {\bf Schramm-Loewner evolution} on $\H_2$ with parameter $\ka$ is the continuously increasing hull process given by the Loewner evolution \eqref{eq:chordal.lde} with random driving function $u_t=\sqrt{\ka}\bm_t$.
\end{dfn}

The $\sle_\ka$ on an arbitrary two-marked domain $D_2$ is the image of $\sle_\ka$ on $\H_2$ under a conformal equivalence $f : \H_2\to D_2$ (and hence is defined only up to a linear time change). More generally, the $\sle_{\ka,\rho}$ processes are generated by Brownian motion with drift; see \cite{\lswrestr,\dubedatdual,\schrammwilson}.

Unfortunately, there are continuously increasing hull process $K_t$ not generated by any curve (see e.g.\ \cite[\S4.4]{\lawler}). That the hull process corresponding to $u_t=\sqrt{\ka}\bm_t$ \emph{is} almost surely generated by a curve was proved for $\ka=8$ in \cite{\lswlerw} and for $\ka\ne8$ in \cite{\rohdeschramm} (see \cite[Ch.~7]{\lawler}). We do not need this fact because we will show that the percolation exploration path converges in law to a conformally invariant random curve satisfying a property which uniquely characterizes $\sle_6$ among the $\sle_\ka$ (hull) processes. This implies {\it a fortiori} that $\sle_6$ is generated by a curve.

\subsection{\texorpdfstring{Characterization of $\sle_6$}{Characterization of SLE(6)}}
\label{ssec:sle.six}

We now present the characterization of $\sle_6$ used by Camia and Newman in identifying the scaling limit of the exploration path.

First, the following result, noted by Schramm in \cite[\S1]{\schramm} and proved in \cite{\lswintersecI}, indicates why $\sle_6$ is the natural candidate. Let $N = \H\cap N_0$ for $N_0$ an open neighborhood of $0$, and let $\Phi$ be a locally real conformal map of $N$ into $\H$ (that is, in some neighborhood of $0$, $\Phi$ has a power series expansion with real coefficients). Consider running an $\sle_\ka$ process $\gam$ until the first time $\tau$ that it leaves $N$: the process is said to have the {\bf locality} property if $\Phi\gam$ (before time $\tau$) again has the law of $\sle_\ka$ up to a time change.

\begin{thm}[Lawler, Schramm, Werner \cite{\lswintersecI,\lswrestr}] \label{thm:local}
The only $\sle_\ka$ with the locality property is $\sle_6$.

\begin{proof}
We follow the proof of \cite{\lawler}. For $\gam$ an $\sle_\ka$ process, let $K_t$ denote the filling of $\gam[0,t]$, and let $N^*=\Phi N$, $\gam^*=\Phi\gam$, $K_t^*=\Phi K_t$, $b^*(t)=\hcap\gam^*[0,t]$. Let $g_t^*$ denote the conformal map $\H\setminus K_t^*\to\H$ with the hydrodynamic normalization, and let $\Phi_t \equiv g_t^* \circ \Phi \circ g_t^{-1}$, so that
$$\dot g_t^*(z) = \f{\dot b^*(t)}{g_t^*(z)-u_t^*},\quad g_0^*(z)=z,$$
with $u_t^*=\Phi_t(u_t)$.

We leave it to the reader to verify that $\dot b^*(t) = 2 \Phi_t'(u_t)^2$. Therefore, we can easily calculate that for $z\in\H$,
$$\dot \Phi_t(z)
= 2\left[ \f{\Phi_t'(u_t)^2}{\Phi_t(z)-u_t^*} - \f{\Phi_t'(z)}{z-u_t} \right].$$
Letting $z\to u_t$, we find $\dot \Phi_t(u_t) = -3\Phi''(u_t)$. Thus, by It\=o's lemma,
\begin{align*}
du_t^* &= \dot \Phi_t(u_t) \ dt + \Phi_t'(u_t) \ du_t + \f{1}{2} \Phi_t''(u_t) \ka \ dt \\
&= \left[ -3+\f{\ka}{2} \right] \Phi_t''(u_t) \ dt + \Phi_t'(u_t) du_t,
\end{align*}
which is a martingale if and only if $\ka=6$. In this case it follows that $\gam^*$ is again an $\sle_6$ process.
\end{proof}
\end{thm}

Any scaling limit of the percolation exploration path is certainly expected to have the locality property, so (heuristically) Thm.~\ref{thm:local} already distinguishes $\sle_6$ as the only possible candidate for the scaling limit of the exploration path.

To rigorously identify the scaling limit with $\sle_6$ we take a slightly different approach: recall \S\ref{sec:bm.harm} where we noted that from the conformal invariance of the Brownian hitting distribution, one can deduce the conformal invariance of the entire Brownian trajectory (up to reparametrization) by taking polygonal approximations. We now apply this idea to give a characterization of $\sle_6$ which we will then verify in the percolation scaling limit.

By conformal invariance it suffices to characterize $\sle_6$ in $\H_2$. For $\gam$ any self-avoiding curve traveling chordally in $\H_2$ and parametrized by $\hcap$, define
$$\tau(\gam,\ep) \equiv\inf\set{t\ge0 : \gam(t)\notin\ball{\ep}};$$
note $\tau(\gam,\ep)\le\ep/2$. Let $\tau_0=0$, and for $j\ge1$ let
$$\tau_j = \tau_{j-1} + \tau(\bar g_{\tau_{j-1}} \gam,\ep),\quad
\bar g_t(z) \equiv g_t(z)-u_t.$$
Define the polygonal approximation $\gam_\ep$ by setting $\gam_\ep(\tau_j)=\gam(\tau_j)$ and interpolating linearly in between. Since $\gam$ is continuous on $[0,\infty]$ it is uniformly continuous, and since $\tau_j-\tau_{j-1}\le\ep/2$ for all $j$ it follows that $\gam_\ep\to\gam$ uniformly as $\ep\decto0$. Thus a deterministic $\gam$ is characterized by its polygonal approximations $\gam_\ep$, and a random $\gam$ is characterized by the laws of these approximations.

For fixed $\ep>0$, call $\wt K_j \equiv \bar g_{\tau_{j-1}} K_{\tau_j}$ ($j\ge1$) the {\bf $\ep$-filling sequence} for $\gam$. Then $\bar g_{\tau_{j-1}}\gam_\ep(\tau_j)$ is the unique point of intersection of $\wt K_j$ with $\circl{\ep}$, and by composition of conformal mappings $\gam_\ep$ is determined by the $\ep$-filling sequence. But if $\gam$ is a random $\sle$ curve, the domain Markov property implies that its $\ep$-filling sequence is \emph{i.i.d.}, characterized by the law of $\wt K_1$.

$\sle_6$ is the unique $\sle_\ka$ for which this law is given by Cardy's formula, another fact noted in \cite[\S1]{\schramm} and proved in \cite{\lswintersecI}. We will formulate this result as follows: in a 2-marked domain $D_2\equiv(D;a,b)$ let $J$ be a simple curve joining $c\in\ol{ab}$, $d\in\ol{ba}$. We refer to $J$ as a {\bf crosscut}; it separates $D$ into subdomains $D_0$ (incident to $a$) and $D_\infty$ (incident to $b$), and we consider running a self-avoiding curve from $a$ until it exits $D_0$, resulting in a hull $K(D_2,J)$. Let $\al_c$ be a simple curve traveling from
$\al_{c0}\in\ol{ac}$ to $\al_{c1}\in J$,
and let $\al_d$ be a simple curve traveling from
$\al_{d0}\in\ol{da}$ to $\al_{d1}\in J$,
with $\al_{c}\cap\al_{d}=\emptyset$. Let $D_0^\al$ denote the connected component of $D_0\setminus(\al_{c}\cup\al_{d})$ between $\al_c$ and $\al_d$. (See Fig.~\ref{fig:crosscut}.) Specifying the law of the hull $K$ is equivalent to specifying the law of its boundary, regarded as an element of $\crv{D}$. Excluding events of probability zero, the $\si$-algebra is generated by the events
$$E_\al\equiv \set{K : K\subseteq \clos{D_0^\al}}.$$

\begin{figure}
\centering
\includegraphics[height=0.56\textwidth]{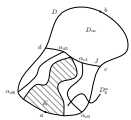}
\caption{Crosscut and hull distribution}
\label{fig:crosscut}
\end{figure}

\begin{thm}[Lawler, Schramm, Werner {\cite[Thm.~3.2]{\lswintersecI}}]
For the setting described above, $\sle_6$ is the unique $\sle_\ka$ for which the law of $K(D_2,J)$ is determined by Cardy's formula:
\beq \label{eq:cardy.hull}
\P^{\sle_6}(E_\al) = \Phi( D_0^\al;a,\al_{c0},\al_{c1},\al_{d0} )
    -\Phi( D_0^\al;a,\al_{c0},\al_{c1},\al_{d0} ).
\eeq
\end{thm}

We omit the proof and refer the reader to \cite{\lswintersecI} and \cite[\S6.7]{\lawler}. The proof for the equilateral triangle (which is sufficient for our purposes) is particularly simple and may be found in \cite[\S6.8]{\lawler} or \cite[\S3.8]{\wernerparkcity}.

The main consequence of this discussion is the following result which will be used in \S\ref{ch:explore} to identify the scaling limit of the exploration path.

\begin{ppn} \label{ppn:char.sle6}
If $\gam$ is a self-avoiding curve such that for all $\ep>0$ the $\ep$-filling sequence is i.i.d.\ with law given by Cardy's formula \eqref{eq:cardy.hull}, then $\gam$ is $\sle_6$.
\end{ppn}

\subsection{Capacity estimates for the LDE}
\label{sec:lde.sle.cap}

In this section we prove Propn.~\ref{ppn:hcap.err} and Propn.~\ref{ppn:lde.distort}, completing the derivation of the LDE. We place both results under the heading of ``capacity estimates'' because they are proved by estimating hitting probabilities of Brownian motion.

\begin{lem}
Let $D\equiv\H\setminus\clos\D$, and let $\bm_t$ be Brownian motion started in $D$ and stopped at $\tau\equiv\tau_D$. If $p(z,w)$ denotes the density of $W_\tau$ on $\pd D$ with respect to Lebesgue measure, then
\beq \label{eq:bm.sin}
p(z,e^{i\thet})
=\f{2}{\pi} \f{\imag z}{|z|^2} \sin\thet (1 + O(|z|^{-1})).
\eeq
Consequently, if $\rad A\le1$ then
\beq \label{eq:hcap.sin}
\hcap A = \f{2}{\pi} \int_0^\pi
    \E_{e^{i\thet}}[\imag W_{\tau_{\H\setminus A}}] \sin\thet\d\thet.
\eeq

\begin{proof}
Recall the Poisson integral formula for $\H$ given in \eqref{eq:pois.h}. If $f$ is a bounded continuous function on $\pd D$ with $f=0$ outside of $\pd\D$, then
$$\E_z[f(W_\tau)]
= P_D f(z) = P_\H(f\circ\vph^{-1})(\vph(z)),$$
where $\vph(z)=z+1/z$ maps $D$ conformally onto $\H$. In particular $\vph$ maps $e^{i\thet}$ to $2\cos\thet\in[-2,2]$, so
\begin{align*}
\E_z[f(W_\tau)]
&= \f{1}{\pi} \int_{-2}^2 Q_{\imag\vph(z)}(t-\real\vph(z))
    f \circ\phi^{-1}(t) \d t\\
&= \f{2}{\pi} \int_0^\pi Q_{\imag\vph(z)}(2\cos\thet-\real\vph(z))
    f(e^{i\thet}) \sin\thet\d\thet.
\end{align*}
\eqref{eq:bm.sin} then follows since $\real\vph(z) = (1+|z|^{-2})\real z$ and $\imag\vph(z)=(1-|z|^{-2})\imag z$. Next, by Propn.~\ref{ppn:hcap.bm} and the strong Markov property for Brownian motion, if $\rad A\le 1$ and $|z|>1$ then
\beq \label{eq:gA.err.sin}
\imag[z-g_A(z)]
= \f{2}{\pi} \f{\imag z}{|z|^2} \int_0^\pi
    \E_{e^{i\thet}}[\imag W_{\tau_{\H\setminus A}}]
    \sin\thet (1+O(|z|^{-2}))\d\thet.
\eeq
Taking $z=iy$ and taking $y\incto\infty$ gives \eqref{eq:hcap.sin}.
\end{proof}
\end{lem}

\begin{proof}[Proof of Propn.~\ref{ppn:hcap.err}]
By scaling we may assume $\rad A = 1$. For $z\in\H\setminus A$, let $h(z)\equiv g_A(z)-z-\hcap A/z$ and consider
$$v(z)\equiv \imag h(z)
= \imag [g_A(z)-z] + \hcap A  \frac{\imag z}{|z|^2}.$$
By \eqref{eq:hcap.sin} and \eqref{eq:gA.err.sin}, there exists a constant $c<\infty$ such that
$$|v(z)|
\le c \f{\imag z}{|z|^3} \hcap A.$$
On the other hand, if $\clos{\ball[z]{r}}\subset\H\setminus A$, the Poisson integral formula \eqref{eq:pois.d} for the disc gives
$$v(w)
= \f{1}{2\pi}\int_0^{2\pi}
    \real\lp\f{R+(w-z)}{R-(w-z)} \rp v(z+Re^{i\vph})\d\vph,
\quad w\in\ball[z]{r},$$
and differentiating we find that for $z$ sufficiently far away from $A$ (say $|z|\ge2$),
$$\absb{\frac{\pd v}{\pd x}(z)}, \
\absb{\frac{\pd v}{\pd y}(z)} \le c \frac{\hcap A}{|z|^3}.$$
It follows by the Cauchy-Riemann equations that $|h'(z)|\le c(\hcap A)/|z|^3$, and since $\lim_{r \to \infty} h(re^{i\thet}) = 0$ we can integrate from $\infty$ to bound $h(re^{i\thet})$:
$$|h(re^{i\thet})| = \left| \int_y^\infty h'(te^{i\thet}) \ dt \right| \le c \frac{\hcap A}{r^2},$$
which concludes the proof.
\end{proof}

It remains to prove Propn.~\ref{ppn:lde.distort}, which is done by relating the size of sets to hitting probabilities of Brownian motion. For $B$ a subset of $\clos\H$ and $A\in\cQ$, define the {\bf capacity of $B$ relative to $A$}
by
$$\capacity_\H(B;A)\equiv\lim_{y\to\infty} y\P_{iy}
    (\bm[0,\tau_{\H\setminus A}] \cap B \ne\emptyset).$$
We call $\capacity_\H(B)\equiv\capacity_\H(B;\emptyset)$ the {\bf capacity} of $B$. This capacity scales as $\capacity_\H rB=r\capacity_\H B$, and has the following invariance property: for $B'\subseteq\H$, let $g_A(B')$ denote the pre-image of $B'$ under the continuation of $g_A^{-1}$ to $\clos\H$. Then:

\begin{lem} \label{lem:cap}
For $A\in\cQ$, $\capacity_\H A = \capacity_\H \wh A=\pi^{-1}\length \wh A$ where $\wh A = g_A(\clos A)$ is an interval in $\R$.

\begin{proof}
Let $\si\equiv\tau_\H$. If $\wh A$ is any interval in $\R$, the Poisson integral formula \eqref{eq:pois.h} for $\H$ gives
$$y\P_{iy}(\bm_\si\in \wh A)
= \f{1}{\pi} \int_{\wh A} \f{y^2}{(t-x)^2+y^2}\d t,$$
which tends to $\pi^{-1}\length \wh A$ as $y\to\infty$. For $A\in\cQ$, by the conformal invariance of Brownian motion,
$\capacity_\H \clos A
= \lim_{y\to\infty} y \P_{g_A(iy)}(\bm_\si \in \wh A)$. By the hydrodynamic normalization $g_A(iy) = u_A(iy)+i v_A(iy) = iy + O(y^{-1})$, therefore
$$\capacity_\H A = \lim_{y\to\infty}
    y \P_{v_A(iy)}(\bm_\si \in \wh A - u_A(iy))
= \lim_{y\to\infty} y \P_{iy}(\bm_\si\in \wh A)
= \capacity_\H \wh A,$$
which concludes the proof.
\end{proof}
\end{lem}

An immediate consequence is that $\capacity_\H A \asymp \diam A$, so to prove Propn.~\ref{ppn:lde.distort} it suffices to estimate the capacity of $g_s(\gam(s,t])$. To this end, here is a simple consequence of Lem.~\ref{lem:cap}:

\begin{cor} \label{cor:cap}
Let $A\in\cQ$, $B\subset\clos\H$ such that $A'\equiv (A \cup B)\cap\H \in\cQ$. Then $\capacity_\H(B;A)=\capacity_\H g_A(B)$.

\begin{proof}
Let $\si\equiv\tau_\H$. By the conformal invariance of Brownian motion,
$$\capacity_\H(B;A)
= \lim_{y\to\infty} y\P_{ g_{A'}(iy)}
    (\bm_\si \in g_{A'}(B))
= \lim_{y\to\infty} y\P_{iy}
    (\bm_\si \in g_{A'}(B)),$$
where the second identity follows from the calculation of Lem.~\ref{lem:cap}. Since
$g_{A'} = \tilde g \circ g_A$ where $\tilde g=g_{g_A(A' \setminus A)}$, it follows similarly that
$$\capacity_\H g_A(B)
= \lim_{y\to\infty} y \P_{\tilde g(iy)}
    (\bm_\si \in \tilde g \circ g_A(B))
= \lim_{y\to\infty} y\P_{iy}
    (\bm_\si\in g_{A'}(B)),$$
which proves $\capacity_\H(B;A)=\capacity_\H g_A(B)$.
\end{proof}
\end{cor}

The intersection probabilities of Brownian motion can be estimated by the following very useful result:

\begin{thm}[Beurling estimate] \label{thm:beurling}
Suppose $0<r_1<r_2<\infty$ and $\gam:[0,1]\to\C$ is a curve from $\circl{r_1}$ to $\circl{r_2}$. Then
\begin{align*}
&\P_z(\bm[0,\tau_{\ball{r_2}}]\cap\gam=\emptyset)
\lesssim (r_1/r_2)^{1/2},\quad \abs{z}\le r_1,\\
&\P_z(\bm[0,\tau_{\C\setminus\clos{\ball{r_1}}}]\cap\gam=\emptyset)
\lesssim (r_1/r_2)^{1/2},\quad \abs{z}\ge r_2.
\end{align*}
\end{thm}

For a proof see \cite[\S3.8]{\lawler}. This estimate allows us to prove the main lemma:

\begin{proof}[Proof of Propn.~\ref{ppn:lde.distort}]
Write $d\equiv\diam\gam$ and $r\equiv\sup\set{\imag z : z\in B}$: by scaling it suffices to show $\capacity_\H g_A(B)\lesssim d^{1/2}$ when $r=1$, $d<1/2$ (say). By Cor.~\ref{cor:cap},
$$\capacity_\H g_A(B)
=\capacity_\H(B;A)
=\lim_{y\to\infty} y\P_{iy}(W[0,\tau]\cap B\ne\emptyset),\quad
    \tau\equiv\tau_{\H\setminus A}.$$
Fix $z_0\in B$. For large $y$, $\P_{iy}(\bm[0,\tau]\cap \ball[z_0]{1}\ne\emptyset)\lesssim c/y$ (e.g.\ by Lem.~\ref{lem:cap}). Once the Brownian motion enters $\ball[z_0]{1}$, the probability that it intersects $B$ before leaving $\H\setminus A$ is bounded above by the probability that it enters $\ball[z_0]{2d}$ before leaving $\H\setminus A$. This is $\lesssim d^{1/2}$ by Thm.~\ref{thm:beurling} so the result follows.
\end{proof}

The results of this section concerned the hypothetical scaling limits of random curves arising in discrete physical models such as LERW and percolation. We return in the next two sections to the proof of Thm.~\ref{thm:explore}.

\section{Scaling limits}
\label{ch:scaling}

In this section we present a result of Aizenman and Burchard \cite{\abholder} which guarantees the existence of \emph{subsequential} weak limits for the percolation exploration paths $\gam^\de$ in discrete domains $D_2^\de$ converging to a limit $D_2$. The next section presents the work of Smirnov and Camia and Newman which pins down the limit to be $\sle_6$. Throughout this section we assume that $D$ is a bounded domain in $\R^d$; without loss $\diam D\le 1$.

\subsection{Systems of random curves}

Recall Defn.~\ref{dfn:curve} of the space of curves $\cC$ with the uniform metric $\dcurve$; this is a complete separable metric space. Given a bounded domain $D$, let $\crv{D}$ denote the subspace of $\cC$ of curves traveling in the closure of $D$. To formalize the notion of a ``discrete curve,'' let $\crv{D}^\de$ denote the subspace of $\crv{D}$ of polygonal curves of step size $\de$.

\begin{dfn}
A {\bf curve configuration} is a closed subset of $\crv{D}$; it is a {\bf $\de$-curve configuration} if it is contained entirely in $\crv{D}^\de$. Denote the space of curve configurations by $\crvset{D}$ and the space of $\de$-curve configurations by $\crvset{D}^\de$.
\end{dfn}

A metric on $\crvset{D}$ is given by the Hausdorff metric induced by the metric $\dcurve$ on $\crv{D}$: for $F,F'\in\crvset{D}$, the Hausdorff distance $\dhaus(F,F')$ between them is the smallest $\ep>0$ such that each is contained in the $\ep$-neighborhood of the other (with respect to $\dcurve$). With this metric, the completeness and separability of $\crv{D}$ is passed on to $\crvset{D}$.

In this section we study {\bf (random) systems of configurations}, collections $F= (F^\de)_{0<\de\le\de_0}$ of random variables where the law $\P^\de$ of $F^\de$ is a probability measure on $\crvset{D}$ (with the Borel $\si$-algebra) with support contained in $\crvset{D}^\de$.

The aim is to specify a regularity condition, verifiable in the percolation setting, which implies precompactness of $\set{\P^\de}$ in the weak topology: this means that along any sequence $\de\to0$ we can extract a subsequence $\de_n$ with $\P^{\de_n}$ converging weakly to a limiting measure on $\crvset{D}$. In view of Thm.~\ref{thm:rsw}, it is natural to seek a regularity condition formulated in terms of crossings of annuli. Let $\ann{x}{r}{R} \equiv \ball[x]{R}\setminus\clos{\ball[x]{r}}$.

\begin{dfn} \label{eq:ker.cross}
Let $k\in\N$, $0<r<1$, $\eta>0$, $x\in \clos D$. A curve $\gam\in\crv{D}$ has a {\bf $k$-fold crossing of power $\eta$ and scale $r$} (for short, a $(k,\eta,r)$-crossing) at $x$ if the annulus
$\ann{x}{r^{1+\eta}}{r}$ is traversed by $k$ separate segments of $\gam$.
\end{dfn}

Let $E(k,\eta,r)$ denote the set of $F\in\crvset{D}$ exhibiting a $(k,\eta,r)$-crossing (i.e., such that some $\gam\in F$ has a $(k,\eta,r)$-crossing).

\begin{dfn}
Let $(F^\de)_{0<\de\le\de_0}$ be a system of configurations specified by laws $(\P^\de)_{0<\de\le\de_0}$. We say the system has {\bf uniform crossing exponents} if for all $k\in\N$ there exist constants $c_k,\lm_k<\infty$ with $\lm_k\to\infty$ such that for any $0<r<1$,
$$\P^\de[E(k,\eta,r)] \le c_k r^{-\lm_k\eta}.$$
\end{dfn}

\begin{thm}[Aizenman, Burchard \cite{\abholder}] \label{thm:ab}
Let $(F^\de)_{0<\de\le\de_0}$ be a system of configurations specified by laws $(\P^\de)_{0<\de\le\de_0}$. If the system has uniform crossing exponents, then the family $\set{\P^\de}$ is precompact.
\end{thm}

Here is the application to our problem of interest: let $F^\de$ be the percolation configuration on the discrete domain $D_2^\de$, regarded as the collection of all interface curves not including the domain boundary.

\begin{cor} \label{cor:ab}
Let $D_2^\de$ be discrete domains converging uniformly to the two-marked domain $D_2$ (with the spherical metric). If $\P^\de$ is the law of the percolation configuration $F^\de$ then the family $\set{\P^\de}$ is precompact.

\begin{proof}
Thm.~\ref{thm:rsw} implies the existence of a $1$-arm crossing exponent $\lm_1$, and the BK inequality implies the $k$-arm crossing exponent $\lm_k=k\lm_1$, so the result follows from Thm.~\ref{thm:ab}.
\end{proof}
\end{cor}

Each $F^\de$ has a distinguished curve $\gam^\de$ (the exploration path), and by applying Thm.~\ref{thm:ab} again we can extract a further subsequence along which $\gam^\de$ converges weakly to a limit curve $\gam$. Consequently, to prove Thm.~\ref{thm:explore} it remains to identify any such limit $\gam$ as an $\sle_6$ curve. This will be done in the next section using the characterization of Propn.~\ref{ppn:char.sle6}.

We turn now to the proof of Thm.~\ref{thm:ab}. By Prohorov's theorem (see e.g.\ \cite{\billingsleyc}) it suffices to show that the family is tight, i.e.\ that for all $\ep>0$ there exists $\cK\subset\crvset{D}$ compact such that $\P^\de(\cK)\ge1-\ep$ for all $\de$. Compactness in $\crvset{D}$ can be characterized as follows:

\begin{lem}
If $\cK$ is a closed subset of $\crvset{D}$ such that the union of all $K\in\cK$ is contained in some $K_0\subset\crv{D}$ compact, then $\cK$ is compact.

\begin{proof}
By elementary topology, if $(S,d)$ is a compact metric space then the space $\cH(S)$ of closed subsets of $S$ taken with the Hausdorff metric $\dhaus$ is also compact (see e.g.\ \cite[p.~280]{munkres}). Under the hypothesis, $\cK$ is a closed subset of the compact space $(\cH(K_0),\dhaus)$ and hence is compact.
\end{proof}
\end{lem}

Compactness in $\crv{D}$ is in turn characterized by the Arzel\`a-Ascoli theorem, so to prove tightness it suffices to prove an equicontinuity bound which holds with probability $\ge1-\ep$ under each $\P^\de$. Towards this end, the next section relates H\"older continuity to annuli crossings.

\subsection{H\"older, tortuosity, and dimension bounds}
\label{sec:single.curve}

The {\bf optimal H\"older exponent} of a curve $\gam$, denoted $\al(\gam)$, is the supremum of all $\al>0$ such that $\gam$ admits a parametrization which is H\"older continuous with exponent $\al$ --- i.e., such that
$$\abs{\gam(s)-\gam(t)} \le c_\al \abs{s-t}^\al
    \quad\forall s,t\in[0,1]$$
for some constant $c_\al<\infty$. Inverting this relation gives the equivalent condition
\beq \label{eq:holder.inv}
|s-t|\ge c_\al' |\gam(s)-\gam(t)|^{1/\al}
    \quad\forall s,t\in[0,1].
\eeq
That is, if two points on the curve are a certain distance apart, H\"older continuity puts a \emph{lower} bound on their time difference. Motivated by this observation, let $M_r(\gam)$ denote the minimal $n$ such that $\gam$ can be partitioned into $n$ segments of diameter $\le r$; this is a measure of the curve's ``tortuosity.'' The {\bf tortuosity exponent} of $\gam$ is defined as
$$\tau(\gam)
\equiv\inf\setb{s>0 : \lim_{r\decto0} r^s M_r(\gam)=0}$$

\begin{ppn} \label{ppn:holder.tort}
For $\gam\in\crv{D}$, $\al(\gam)=1/\tau(\gam)$.

\begin{proof}
If $\psi$ is an increasing homeomorphism of $[0,1]$ such that $\psi(|\gam(t_1) - \gam(t_2)|) \le |t_1-t_2|$ whenever $|\gam(t_1)-\gam(t_2)| \le 1$, then certainly $M_r(\gam) \le \ceil{1/\psi(r)}$ for all $0<r\le 1$, and it follows that $\tau(\gam)\le\al(\gam)^{-1}$.

Suppose conversely that $M_r(\gam)\le 1/\psi(r)$ for $0<r\le1$. For a curve $\eta\in\crv{D}$, define
$$\tilde t(\gam) \equiv \sum_{n\ge0} \f{\psi(r_n) M_{r_n}(\gam)}{(n+1)^2},
\quad r_n\equiv 2^{-n}.$$
Take an auxiliary parametrization $\gam_0$ for $\gam$, and consider
$$t(u)\equiv \f{\tilde t(\gam_u)}{\tilde t(\gam)},\quad 0\le u\le 1.$$
(Notice $\tilde t(\gam)\le 2$ by the assumed bound on $M_r(\gam)$). Then $t:[0,1]\to[0,1]$ is a strictly increasing right-continuous function. Its generalized inverse $u$ is continuous, so we may reparametrize $\gam(t)=\gam_0(u(t))$. Let $r=\abs{\gam(s)-\gam(t)}$ for $s<t$: then $M_{r_n}(\gam_t)-M_{r_n}(\gam_s)\ge1$ for all $r_n\le r$, i.e.\ for all $n \ge \log_2(1/r)$, so
$$t-s \ge \f{\psi(r)}{2(\log_2(1/r)+1)^2}.$$
This implies $\al(\gam)\ge\tau(\gam)^{-1}$ which concludes the proof.
\end{proof}
\end{ppn}

Tortuosity in general is difficult to compute or estimate, but under regularity conditions $M_r(\gam)$ can be bounded by quantities which depend only on the curve's {\bf trace}, the set $\set{\gam(t):t\in[0,1]}$. For example, let $N_r(\gam)$ denote the minimal $n$ such that the trace of $\gam$ can be covered by $n$ sets of diameter $r$. The {\bf upper box dimension} or {\bf Minkowski dimension} of the curve is
$$\boxdim(\gam)
\equiv\inf\setb{ s>0 : \lim_{r\decto0} r^s N_r(\gam)=0 }.$$
Trivially $N_r(\gam)\le M_r(\gam)$ so $\boxdim\le\tau$. Further it is easy to see that if $r'<r$ then
$$N_{r'}(\gam) \le N_r(\gam) \ceil{r/r'}^d,$$
so $\boxdim\le d$ while $\tau$ is unbounded.

Recall Defn.~\ref{eq:ker.cross} of a $(k,\eta,r)$-crossing. Let us say that a curve $\gam$ has the {\bf tempered crossing property} if for all $0<\eta<\eta_0$ there exist $k,r_0$ (both depending on $\eta$) such that $\gam$ exhibits no $(k,\eta,r)$-crossings.

\begin{ppn} \label{ppn:tempered}
If $\gam\in\crv{D}$ exhibits no $(k,\eta,r)$-crossings then
\beq \label{eq:tort.boxdim}
M_{2r}(\gam) \le k N_{r^{1+\eta}}(\gam).
\eeq
Consequently, if $\gam$ has the tempered crossing property then $\boxdim(\gam)=\tau(\gam)$.

\begin{proof}
We already noted $\boxdim\le\tau$ so it remains to prove the reverse inequality. Fix $\eta>0$ and let $k,r$ be as given by the tempered crossing property. Partition the curve as follows: set $t_0=0$, and for $j>0$ let
$$t_j = \inf\set{t>t_{j-1} :
    \abs{\gam(t)-\gam(t_{j-1})}}\ge 2r,$$
provided this time is well-defined. If $\gam$ does not leave $\ball[\gam(t_{j-1})]{2r}$, terminate by setting $t_j=1$.

The number of segments is an upper bound for $M_{2r}(\gam)$. Consider any covering of $\gam$ by balls of radius $r^{1+\eta}$: since $\gam$ has no $(k,\eta,r)$-crossings and $\abs{\gam(t_j)-\gam(t_{j-1})}\ge 2r$, each ball can contain at most $k$ of the points $\gam(t_j)$. This implies \eqref{eq:tort.boxdim}, and so $\tau(\gam) \le (1+\eta)\boxdim(\gam)$. The result for a curve with the tempered crossing property follows by taking $\eta\decto0$.
\end{proof}
\end{ppn}
\subsection{Proof of tightness}

We now prove Thm.~\ref{thm:ab}. Recall $(F^\de)_{0<\de\le\de_0}$ is a system of configurations specified by laws $(\P^\de)_{0<\de\le\de_0}$.

\begin{lem} \label{lem:tempered.whp}
Under the hypotheses of Thm.~\ref{thm:ab}, for all $\ep,\eta>0$ there exist $k,r_0$ such that
$$\P^\de[ F^\de \in E(k,\eta,r) \text{ for some } r\le r_0 ] <\ep$$
for all $0<\de\le\de_0$.

\begin{proof}
Suppose $F^\de$ has a $(k,\eta,r)$-crossing at $x_0\in D$ for $r\le 4^{-1/\eta}$. If $x$ is any point in $\ball[x_0]{r^{1+\eta}/2}$, then $F^\de$ exhibits a $k$-fold crossing of $\ann{x}{2r^{1+\eta}}{r/2}$, hence also of $\ann{x}{2r_n^{1+\eta}}{r_{n+1}/2}$ where $r_{n+1}<r\le r_n$ ($r_n\equiv 2^{-n}$). Applying the crossing exponent hypothesis for shells $\ann{x}{2r_n^{1+\eta}}{r_{n+1}/2}$ centered at an $r_n^{1+\eta}/2$-net of points $x\in D$ gives
$$\P^\de[ F^\de \in E(k,\eta,r) \text{ for some } r_{n+1}<r\le r_n ]
\le c_k' \f{r^{\eta \lm_k}}{r^{(1+\eta)d}}.$$
Choose $k$ large enough (depending on $\eta$) so that the exponent on $r$ is positive, and sum over scales $r_n \le r_0$ to conclude the result.
\end{proof}
\end{lem}

Define the random variables
$$N^\de(r,R)
\equiv \sup\set{ N_r(\gam) : \gam\in F^\de, \diam\gam\ge R }.$$
For each $\de$, $N^\de(r,R)$ is nonincreasing in $r,R$.

\begin{lem} \label{lem:boxdim.stoch.bd}
Under the hypotheses of Thm.~\ref{thm:ab}, the random variables
$$X^\de\equiv \sup_R
    \sup_{r\le R} \f{N^\de(r,R) r^d}{(r/R)^{\lm}} \ell(r),$$
where $\lm\equiv\lm_1$ and $\ell$ is a polylogarithmic factor, are stochastically bounded.

\begin{proof}
Since $D$ is a bounded domain, it suffices to obtain a bound over all scales $0<r\le R\le 1$. For $r>0$ let $\Pi_r$ denote the standard grid partition of $\R^d$ into rectangles of diameter $r$. For $r\le R$ let $\wt N^\de(r,R)$ denote the number of sets $B\in\Pi_r$ meeting some curve in $F^\de$ of diameter $\ge R$; clearly $N^\de(r,R)\le \wt N^\de(r,R)$.

For each $B\in\Pi_r$ let $x_B$ be a point such that $B \subseteq \ball[x_B]{r/2}$. If $B \in \Pi_\ell$ meets a curve of diameter $\ge R$, we must have a crossing of the spherical shell $\ann{x_B}{r/2}{R/2}$, which by hypothesis occurs with probability $\le c_1 (r/R)^\lm$. Summing over the $\asymp r^{-d}$ sets in $\Pi_r$ gives
$$\E_\de[\wt N^\de(r,R)] \le c r^{-d} (r/R)^\lm.$$
To obtain a bound over all scales, let
$$U^\de
\equiv \sum_{m\ge0} \f{1}{(m+1)^2}
    \sum_{n\ge m} \f{1}{(n+1)^2}
    \f{\wt N^\de_{r_n,R_m}}{\E[\wt N^\de_{r_n,R_m}]},\quad
    r_n\equiv 2^{-n}, \ R_m \equiv 2^{-m}.$$
$\E_\de[U^\de]=1$ so the $U^\de$ are stochastically bounded by Markov's inequality. If $R_m\le R<R_{m-1}$ and $r_n\le r<r_{n-1}$ then
\begin{align*}
&\wt N^\de(r,R) \le \wt N^\de(r_n,R_m)
\le U^\de (n+1)^4 \E[\wt N^\de(r_n,R_m)]\\
&\le U^\de [\log_2(1/r)+2]^4 r^{-d} (r/R)^\lm.
\end{align*}
This bound holds simultaneously for all $r\le R$ so the result follows.
\end{proof}
\end{lem}

\begin{proof}[Proof of Thm.~\ref{thm:ab}]
We can always decrease the $\lm_k$, so assume without loss that $\lm<d$. Let $\ep,\eta>0$, and let $k,r_0$ be as given by Lem.~\ref{lem:tempered.whp}, so that the bound \eqref{eq:tort.boxdim} holds with $\P^\de$-probability $>1-\ep$ for all $0<\de\le\de_0$. Then by Lem.~\ref{lem:boxdim.stoch.bd} there exists $C<\infty$ such that
$$M_r(\gam)
\le \f{C}{ (\diam \gam)^\lm }
    \f{\ell(r^{1+\eta})}{ r^{(1+\eta)(d-\lm)}}
    \quad \forall r\le \diam \gam, \ \forall\gam\in F^\de$$
holds with $\P^\de$-probability $\ge 1-2\ep$ for all $0<\de\le\de_0$. Inverting this (as in the proof of Propn.~\ref{ppn:holder.tort}) gives the H\"older bound
\beq \label{eq:holder.diam}
\abs{\gam(s)-\gam(t)} \le
    C \ell(r) g(\diam\gam) \abs{s-t}^{1/[(1+\eta)(d-\lm)]},
\eeq
where $\ell$ is a (different) polylogarithmic factor and $g(R) \equiv R^{-\lm/[(1+\eta)(d-\lm)]}$. Interpolating between this and the trivial bound $\abs{\gam(s)-\gam(t)}\le\diam\gam$ gives
$$\abs{\gam(s)-\gam(t)} \le
    C \ell(r) \abs{s-t}^{1/[(1+\eta)(d-\lm)+\lm]}.$$
Thus we have found a H\"older continuity bound holding with $\P^\de$-probability $\ge1-2\ep$ for all $0<\de\le\de_0$, and tightness follows by the Arzel\`a-Ascoli theorem.
\end{proof}

\begin{rmk} \label{rmk:minkowski}
Notice that \eqref{eq:holder.diam}, Propn.~\ref{ppn:holder.tort}, and the relation $\boxdim\le\tau$ imply that the limiting curves have Minkowski dimension strictly less than $d$.
\end{rmk}

\section{Limit of the exploration path}
\label{ch:explore}

In this final section we present a proof of Thm.~\ref{thm:explore}. Our exposition is based on the work of Binder, Chayes, and Lei \cite{\bclperci,\bclpercii} and of Camia and Newman \cite{\cmnw}.

By the results of \S\ref{ch:scaling}, the set of laws $\P^\de$ of the percolation configurations (regarded as curve configurations) is precompact: from any sequence $\de\decto0$ we can extract a further sequence along which $\P^\de$ converges weakly, to a limit which depends {\it a priori} on the particular subsequence. It remains therefore to uniquely identify the weak limits, which will be done in \S\ref{ssec:explore.pf} using the characterization of Propn.~\ref{ppn:char.sle6}. Before doing this, however, we need to address some issues concerning discrete approximation for $k$-marked domains which were ignored in \S\ref{ch:perc} (see Rmk.~\ref{rmk:admiss}): this is the topic of \S\ref{ssec:admiss}.

\begin{rmk} \label{rmk:skorohod}
Whenever we have a precompact family of probability measures on a separable space we will assume that we work within a weakly convergent subsequence, and further, by the Skorohod coupling theorem (see Rmk.~\ref{rmk:wk.conv}), that this sequence has an a.s.\ convergent coupling.

Specifically, by Cor.~\ref{cor:ab} and the subsequent comments, we may assume that along our subsequence $F^\de\to F$ with respect to the Hausdorff metric on $\crvset{D}$, and that $\gam^\de\to\gam$ uniformly.
\end{rmk}

\subsection{Admissible domains and discrete approximation}
\label{ssec:admiss}

Recall the notation for $k$-marked domains introduced in Defns.~\ref{dfn:marked.domain} and~\ref{dfn:marked.domain.disc}. Clearly, even if $D_2$ is restricted to be a Jordan domain, to prove this result one needs to consider more generally the ``slit domains'' $(D\setminus K_t;\gam(t),b)$, where $K_t$ is the filling of $\gam[0,t]$ (see \S\ref{sec:lde.sle}). A useful notion here is that of {\bf prime end}, first introduced by Carath\'eodory \cite{\caratheodory}. We omit a formal definition (see \cite{\pommerenke,\epstein}), but roughly speaking a prime end of $D$ is a ``conformal boundary point'' --- it may not be a boundary point itself, but it ``corresponds'' to a boundary point of $\D$ under conformal mappings $\D \to D$. For example, if $D=\D\setminus[0,1]$, the point $1/2$ ``splits'' into two distinct prime ends: any conformal map $\D\to D$ has a continuous extension to the unit circle by Carath\'eodory's theorem, and two distinct points will map to $1/2$.

\begin{dfn}
A {\bf (generalized) $k$-marked domain} means a domain $D$ whose boundary $\pd D$ is a continuous closed curve $\eta:[0,1]\to\C$, with marked prime ends $P_i = \eta(t_i)$ for $0= t_1 \le \cdots \le t_k\le 1$. The domain is {\bf admissible} if $\eta$ is simple on each $[t_i,t_{i+1}]$.
\end{dfn}

A rather more subtle point is that we need in addition a more general notion of discrete approximation than that of Defn.~\ref{dfn:domain.unif.conv}. A natural form of convergence for complex domains is \emph{Carath\'eodory convergence}: for complex domains $D_n,D$ it is said that {\bf $D_n$ converges to $D$ in the Carath\'eodory sense} if
\begin{enumerate}[(i)]
\item $z\in D$ implies $z\in \liminf D_n$, and
\item $z_n \notin D_n$, $z_n\to z$ implies $z \notin D$.
\end{enumerate}
Note that a single sequence can have multiple Carath\'eodory limits, for example the doubly slit domain $\C\setminus((-\infty,1/n] \cup [1/n,\infty))$ converges to both $\H$ and $-\H$.

\newtheorem*{thm:ker}{Carath\'eodory kernel theorem}
\begin{thm:ker}
Let $D_n,D$ complex domains. There exist conformal maps $f_n:\H\to D_n$, $f:\H\to D$ with $f_n\to f$ locally uniformly in $\H$ if and only if $D_n\to D$ in the Carath\'eodory sense.
\end{thm:ker}

However, as noted in \cite{\bclpercii}, Carath\'eodory convergence is insufficient for the verification of Cardy's formula because of boundary issues. This motivates the following (stronger) definition of domain convergence:

\begin{dfn} Let $D_k^\de,D_k$ be admissible $k$-marked domains. We say $D_k^\de$ {\bf converges conformally} to  $D_k$, denoted $D_k^\de \cconv D_k$, if each $A_j^\de$ converges uniformly to a curve $A_j^\star$ traveling from $P_j$ to $P_{j+1}$ such that
\begin{enumerate}[(i)]
\item $A_j \subseteq A_j^\star \subset \C\setminus ( D \cup (\bigcup_{\ell\ne j} A_\ell ))$ (as sets), and
\item both $A_j^\star$ and its reverse are self-avoiding.\footnote{Formally, we mean that there exists a domain $D \subset D' \subset \C\setminus \bigcup_{\ell\ne j} A_\ell$ such that both $A_j^\star$ and its reverse travel chordally in $D'$, and are self-avoiding in the sense of Defn.~\ref{dfn:self.avoiding}.}
\end{enumerate}
\end{dfn}

Conformal convergence is weaker than the uniform convergence of Defn.~\ref{dfn:domain.unif.conv} because we only require $A_j^\star$ to agree with $A_j$ from the ``perspective'' of the domain interior; $A_j^\star$ is allowed to make excursions away from $\pd D$. We leave the reader to verify that the proof of Thm.~\ref{thm:smirnov} also implies the following generalization:

\begin{thm}
\label{thm:smirnov.gen}
Let $D_4^\de$ be discrete four-marked domains converging conformally to the bounded admissible domain $D_4$. Then
$$\Phi(D_4^\de) \to \Phi(D_4),$$ where $\Phi(D_4)$ denotes the evaluation of Cardy's formula for $D_4$.\footnote{Note that Cardy's formula still makes sense because we marked prime ends on the boundary of $D_4$.}
\end{thm}

\subsection{Convergence of slit domains}
\label{ssec:explore.pf}

We turn finally to the verification of the conditions of Propn.~\ref{ppn:char.sle6}, which will be done in an inductive manner as follows: let $D_2^\de$ be discrete domains converging conformally to the bounded admissible domain $D_2$ (again, here $D$ is thought of as the original domain minus the filling up to some time $t$). Let $J$ be a crosscut of $D_2$ joining $c\in\ol{ab}$ to $d\in\ol{ba}$, and let $K\equiv K(D_2,J)$. Take conformal maps $f^\de:\H_2\to D_2^\de$, $f:\H_2\to D_2$ with $f^\de\to f$ locally uniformly on $\H_2$ as given by the Carath\'eodory kernel theorem, and let $J^\de \equiv f^\de \circ f^{-1}(J)$. Consider running the discrete curve $\gam^\de$ until the first time $\tau^\de$ that it reaches a hexagonal vertex on the other side of $J^\de$ from the initial point. The {\bf discrete filling} $K^\de \equiv K^\de(D_2^\de,J^\de)$ is the smallest simply connected closed set containing the union of all hexagons explored up to time $\tau^\de$.

By Rmk.~\ref{rmk:skorohod}, we always work along a subsequence with $F^\de\to F$ in $\crvset{D}$ and $\gam^\de\to\gam$ uniformly. In fact, we assume (passing to a further subsequence as needed) that the discrete filling boundaries $\pd K^\de$ converge uniformly to a curve $\eta$. Note that $\eta$ need \emph{not} agree with $\pd K$ --- in fact, from our definition of conformal convergence, it need not even lie entirely in $D$! The following is the main content of the inductive step:

\begin{ppn} \label{ppn:admiss.cardy}
In the setting described above, almost surely $\gam,\eta$ are self-avoiding, $(D^\de \setminus K^\de;\gam(\tau^\de),b^\de) \cconv
(D\setminus K;\gam(\tau),b)$, and the law of $K$ is given by Cardy's formula.
\end{ppn}

As a first step, we show that the stopping rules determined by $J$ and $J^\de$ are consistent:

\begin{lem} \label{lem:crosscut}
In the setting of Propn.~\ref{ppn:admiss.cardy}, parametrize $\gam$ by $\hcap$ and parametrize $\gam^\de$ so that $\sup_t \abs{\gam^\de(t)-\gam(t)}\to0$. Then $\tau^\de\to\tau$.

\begin{proof}
If $\tau^\de$ fails to converge to $\tau$ as $\de\decto0$ this means that the curves $\gam^\de$ approach arbitrarily closely a point $z_0\in J^\de$ without crossing, then travel a constant-order distance away (so that capacity increases) before eventually crossing $J^\de$. This implies that for some $\ep>0$ and for all $\ep'<\ep$, if we condition on the first time that $\gam^\de$ comes within distance $\ep'$ of $J^\de$ and let $z_0^\de$ denote the nearest point on $J^\de$, the curve will exit the annulus $\ann{z_0^\de}{\ep'}{\ep}$ without crossing $J^\de$. For fixed $\ep$ this probability decreases to zero as $\ep'\decto0$ (e.g.\ by Thm.~\ref{thm:rsw}), and taking $\ep\decto0$ proves $\tau^\de\to\tau$.
\end{proof}
\end{lem}

We now show that curves in the  (subsequential) limiting percolation configuration are almost surely self-avoiding. The general method for proving such results (see e.g.\ \cite{\adapath}) is via {\it a priori} estimates on crossing events. By a {\bf non-monochromatic $k$-arm crossing} we mean $k$ disjoint crossings not all of the same color. Here is the estimate we will use to control the behavior of $\gam$ in the domain interior:

\begin{lem}[Kesten, Sidoravicius, Zhang {\cite[Lem.~5]{\kszwords}}]
\label{lem:ksz}
The $\P^\de$-probability of a non-monochromatic five-arm crossing of $\ann{z}{r}{R}$ is $\asymp (r/R)^2$ for all $R/r,\de/r \ge c$. Consequently the $\P^\de$-probability of a non-monochromatic six-arm crossing of $\ann{z}{r}{R}$ is $\lesssim (r/R)^{2+\lm}$, where $\lm$ is as in Thm.~\ref{thm:rsw}.
\end{lem}

For the proof of the five-arm exponent see \cite{\kszwords}; the six-arm exponent follows directly from Thm.~\ref{thm:rsw} and the BK inequality. Controlling the behavior of $\gam$ at the domain boundary is more subtle, and the following is the estimate we will require:

\begin{lem}[Lawler, Schramm, Werner {\cite[Appx.~A]{\lswonearm}}]
\label{lem:lsw.three.arm}
Let $\Gam$ be a fixed smooth closed contour inside $D$, and let $z\in \pd D$. Then there exists a constant $c\equiv c(\Gam)$ such that the $\P^\de$-probability of a non-monochromatic three-arm crossing from $\ball[z]{r}$ to $\Gam$ within $D$ is $\le c r^2$.
\end{lem}

This is not precisely the estimate which is proved in \cite{\lswonearm} but we leave the reader to check that it follows from a slight modification of their argument. We turn now to the proof of Propn.~\ref{ppn:admiss.cardy}.

\begin{proof}[Proof of Propn.~\ref{ppn:admiss.cardy}]
Each time $\gam$ intersects a previous part of its path or the boundary of $D$, a region is ``sealed off'' (disconnected from $b$), and to show that $\gam$ is self-avoiding we must show that it does not dive into sealed-off regions.

\begin{enumerate}[(1)]
\item {\it No triple points in interior.} If $\gam$ enters a sealed-off region in the interior of the domain, then $\gam$ must have a triple visit at the entry point. This implies that for some $\ep>0$ and for all $\ep'<\ep$, $\gam^\de$ has a six-fold crossing of some annulus $\ann{z}{\ep'}{\ep}$, $z\in D$, for sufficiently small $\de$. Fig.~\ref{fig:six.arm} shows two topologically distinct ways in which this can occur: although the only situation which concerns us is that of Fig.~\ref{fig:six.arm.dive}, we will eliminate both possibilities. Since $\gam^\de$ has blue to the left and yellow to the right, in all cases the six-fold crossing implies a non-monochromatic six-arm crossing of $\ann{z}{\ep'}{\ep}$. Moreover the annuli can be taken to lie on a square grid of side length $\asymp \ep'$, so summing over $\lesssim 1/(\ep')^2$ annuli and applying Lem.~\ref{lem:ksz} we find that the probability of such an event is $\lesssim (\ep')^\lm/\ep^{2+\lm}$. Taking first $\ep'\decto0$ and then $\ep'\decto0$ shows that there are no interior triple points.

\begin{figure}
\centering
\subfloat[Triple point without dive]{%
\includegraphics[height=0.4\textwidth]{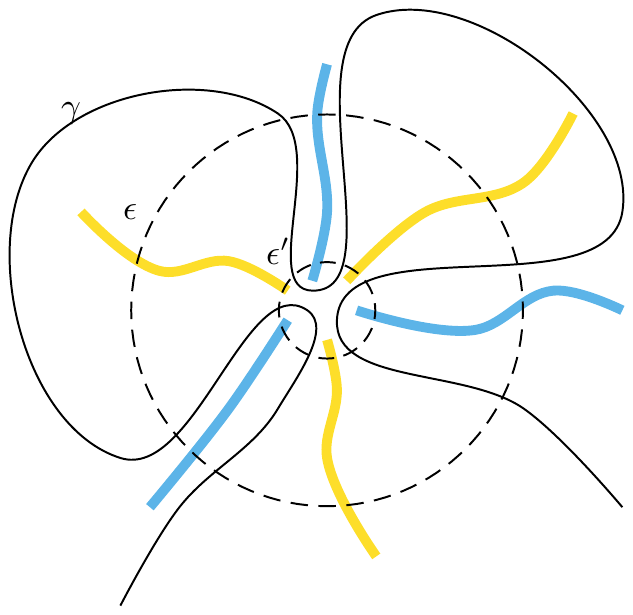}}
\quad
\subfloat[Triple point with dive]{\label{fig:six.arm.dive}%
\includegraphics[height=0.4\textwidth]{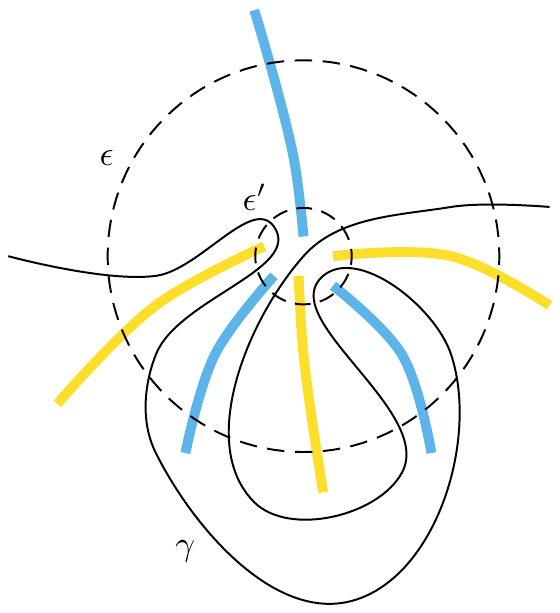}}
\caption{Interior triple point and six-arm event}
\label{fig:six.arm}
\vspace*{-12pt}
\end{figure}

\item {\it No double points on boundary.} If $\gam$ enters a sealed-off region at the domain boundary, then $\gam$ must have a double visit at the entry point, and we claim this does not occur. Let $\vph$ be a conformal map $\D\to D$ and let $\Gam_\ep\equiv\vph(\circl{1-\ep})$; note that $\dist(\Gam_\ep,\pd D)\decto0$ as $\ep\decto0$ by a simple compactness argument. Thus, arguing similarly as above we see that a boundary double point implies that for some $\ep>0$ and all $\ep'<\ep$, there is a non-monochromatic three-arm crossing from an $\ep'$-neighborhood of $\pd D$ to $\Gam_\ep$ for sufficiently small $\de$. But $\pd D$ has Minkowski dimension strictly less than two (Rmk.~\ref{rmk:minkowski}), so the claim follows by partitioning $\pd D$ into segments of diameter $\le\ep'$ and applying Lem.~\ref{lem:lsw.three.arm} with $\Gam = \Gam_\ep$.

\item {\it No retracing.} We claim $\gam$ does not trace any non-trivial segment of the boundary or of its past hull. Retracing implies lengthwise crossings as $\de\decto0$ of four-marked domains which are conformally equivalent to $\ep'\times\ep$ rectangles for fixed $\ep$ and all $\ep'<\ep$. The probability of such a crossing is $\lesssim \al^{\ep/\ep'}$ for some constant $\al<1$. Summing over $\lesssim 1/(\ep')^2$ such domains (covering $\gam$ and $\pd D$) gives the result.
\end{enumerate}
The entire argument above applies equally to $\eta$ so we find that both $\gam,\eta$ are self-avoiding. Combining this with the result of Lem.~\ref{lem:crosscut} gives that
$(D^\de \setminus K^\de;\gam(\tau^\de),b^\de) \cconv
(D\setminus K;\gam(\tau),b)$,
and the law of $K$ is determined by Cardy's formula by Thm.~\ref{thm:smirnov.gen}.
\end{proof}
We now conclude the proof of the main theorem:

\begin{proof}[Proof of Thm.~\ref{thm:explore}]
By Propn.~\ref{ppn:admiss.cardy} the limit $\gam$ is self-avoiding, so it remains to show that the $\ep$-filling sequence satisfies the characterization of Propn.~\ref{ppn:char.sle6}.

Let $f:\H_2\to D_2$, $f^\de : \H_2\to D_2^\de$ be as above. Let $J_1 = f(\circl{\ep}^+)$ and $J_1^\de = f^\de \circ f^{-1}(J_1)$.
Let $x_1^\de$ denote the tip of $K_1^\de$ (the point where $\gam^\de$ crosses $J^\de$), and $x_1$ the tip of $K_1$ (its point of intersection with $J$), and define
$$D_2^\de(1) \equiv (D_2^\de\setminus K_1^\de;x_1^\de,b^\de),\quad
D_2(1) \equiv (D_2\setminus K_1; x_1,b).$$
By Propn.~\ref{ppn:admiss.cardy}, the law of $K_1$ is determined by Cardy's formula, and further $D_2^\de(1)\cconv D_2(1)$ which means that we can repeat the construction in the new admissible domain $D_2(1)$. Continuing in this way the result is proved.
\end{proof}

\section{Conclusion}
\label{ch:end}

Since the introduction of $\sle$ in \cite{\schramm}, a number of discrete interface models have been shown to have scaling limits which are versions of $\sle$. We have focused on a particular model throughout this article, so we conclude with a listing of other examples. The following results are proved:
\begin{itemize}
\item $\sle_2$ is the scaling limit of the loop-erased random walk, and $\sle_8$ is the scaling limit of the uniform spanning tree Peano path (Lawler, Schramm, Werner \cite{\schramm, \lswlerw}).

\item $\sle_6$ is the scaling limit of the percolation exploration path (Smirnov \cite{\smirnovperc, \smirnovpercarxiv}, Camia and Newman \cite{\cmnw}). The scaling limit of the full configuration is determined in \cite{\cmnwfull}.

The outer boundary of $\sle_6$ is $\sle_{8/3}$, which is also the outer boundary of a certain reflected Brownian motion \cite{\lswrestr}; for a discussion of the connection see \cite{\lwuniv}.

\item $\sle_4$ is the scaling limit of the path of the harmonic explorer and of the contour lines of the two-dimensional discrete Gaussian free field (GFF) with certain boundary conditions (Schramm and Sheffield \cite{\ssexplorer,\ssdgff}, extended to more general models by Miller \cite{\millergl,\milleruniv}). There is also a well-defined sense in which $\sle_4$ is a contour line of the two-dimensional continuum GFF, again with certain boundary conditions \cite{\ssgff}. In forthcoming work Miller and Sheffield \cite{\millersheffield} prove that the collection of \emph{all} discrete GFF contours at certain heights converge to $\cle_4$.

\item $\sle_3$ is the scaling limit of interfaces in the critical Ising spin model, and $\sle_{16/3}$ is the scaling limit of interfaces in the corresponding FK cluster representation with parameters $q=2$, $p=1-e^{-2\be_c}$ \cite{\smirnovtowards,\smirnovrci,\smirnovrcii,\ksrciii,\chelkaksmirnovuniv,\chelkaksmirnovising}.
\end{itemize}
The relationship between $\ka$ and $16/\ka$ that is evident in this list is a manifestation of the {\bf Duplantier duality}; see \cite{\dubedatdual,\dubedatdualii,\zhan,\zhanii}.

The following results are conjectured; for more see \cite[\S9]{\rohdeschramm}:
\begin{itemize}
\item Chordal $\sle_{8/3}$ is the scaling limit of the self-avoiding walk in a half-plane $\de\Lm\cap\clos\H$ started at $0$ \cite{\lswrestr}. For recent progress see \cite{\dcssaw}.

\item $\sle_4$ is the scaling limit of loops in the double dimer model. Recent support for this conjecture can be found in \cite{\kwdimers}. Results on the conformal invariance of dimers are in \cite{\kenyon,\kenyongff,\kenyonloops}.

\item The uniform measure on simple grid paths joining a boundary point to an interior point (joining two fixed boundary points) converges to radial (chordal) $\sle_8$ \cite[\S9]{\rohdeschramm}.

\item $\sle_{\ka(q)}$ is the scaling limit of boundaries in the critical FK cluster model with parameters $q$ and $p=p(q)=\sqrt{q}/(1+\sqrt{q})$, where $q \in(0,4)$ and $\ka(q) = 4\pi/\arccos(-\sqrt{q}/2) \in (4,8)$ \cite[\S9]{\rohdeschramm}. (The case $q=2$, corresponding to $\ka=16/3$, has been solved by Smirnov.)

\item $\sle_\ka$ is the scaling limit of a form of the $O(n)$ loop model for $8/3\le\ka\le8$, $n = -2\cos(4\pi/\ka)$ \cite{\kagernienhuis,\sheffieldtreescle}.

\item Loops in the critical XOR-Ising model (the product of two independent critical Ising spin configurations), and more generally in the critical Ashkin-Teller model, converge to a form of $\sle_4$ \cite{\wilsonxor,\ikhlefrajabpourat}.
\end{itemize}

\bibliographystyle{acmtrans-ims}

\end{document}